\documentclass{amsart}
\usepackage[utf8]{inputenc}
\usepackage{amsfonts}
\usepackage{hyperref}
\hypersetup{
pdftitle={},
pdfauthor={},
}
\usepackage{breakurl}

\usepackage{amsmath}
\usepackage{xcolor}
\usepackage{amsthm}
\usepackage{pdflscape}
\usepackage{pgfplots}
\usepackage{mathrsfs}

\newtheorem{thm}{Theorem}[section]
\newtheorem*{thm*}{Theorem}
\newtheorem*{cla*}{Classification}
\newtheorem*{main*}{Main Result}
\newtheorem{cor}[thm]{Corollary}
\newtheorem*{cor*}{Corollary}
\newtheorem{prop}[thm]{Proposition}
\newtheorem{lem}[thm]{Lemma}

\newtheorem{corollary}[thm]{Corollary}

\newtheorem{obs}[thm]{Observation}

\theoremstyle{definition}
\newtheorem{defn}[thm]{Definition}

\newtheorem{ex}[thm]{Example}

\newtheorem{notn}[thm]{Notation}

\newtheorem{remark}[thm]{Remark}
\newtheorem{appr}[thm]{Approach}

\theoremstyle{remark}

\usepackage{amsmath}	
\usepackage{amssymb}		
\usepackage{amsthm}		
\usepackage{setspace}	
\usepackage{array} 
\usepackage{float}

\usepackage[mathscr]{euscript} 
 \let\mathscr\relax
\usepackage[scr]{rsfso}

\newcommand{\la}{\langle}
\newcommand{\ra}{\rangle}

\newcommand{\ZZ}{\mathbb{Z}}

\DeclareMathOperator{\coker}{coker}

\newcolumntype{P}[1]{>{\centering\arraybackslash}p{#1}}		            
\usepackage[top=1in,bottom=1in,left=1.25in,right=1in]{geometry}

\newtheorem*{lemma*}{lemma}

\usepackage{graphicx}

\usepackage[all]{xy}

\newcommand{\onto}{\twoheadrightarrow}

\DeclareMathOperator{\lcm}{lcm}
\usepackage{tikz-cd}
\usetikzlibrary{cd}
\usepackage{adjustbox}
\usepackage{caption}  
\usepackage{subcaption}  
\usepackage{breakurl}

\usepackage[maxbibnames=9]{biblatex}
\newcommand{\Cone}[1]{\operatorname{Cone}\left({#1}\right)}

\newcommand{\FF}{\mathbb{F}}

\newcommand{\GG}{\mathbb{G}}

\DeclareMathOperator{\Img}{Im}

\DeclareMathOperator{\Span}{span}

\usepackage{amsaddr} 

\usepackage{orcidlink} 

\title{DG-Sensitive Pruning \& a Complete Classification of DG Trees and Cycles}

\author{Hugh Geller \orcidlink{0000-0002-4012-6404}}
\address{Center for Naval Analyses, Arlington, Virginia 22201 U.S.A.}
\email{\href{mailto:geller.hugh@gmail.com}{geller.hugh@gmail.com}}

\author{Desiree Martin \orcidlink{0000-0002-9297-582X}}
\address{Mathematics Department, Syracuse University, Syracuse, New York 13244 U.S.A.}
\email{\href{mailto:dmarti02@syr.edu}{dmarti02@syr.edu}}

\author{Henry Potts-Rubin \orcidlink{0000-0002-9864-911X}}
\address{Mathematics Department, Syracuse University, Syracuse, New York 13244 U.S.A.}
\email{\href{mailto:hpottsru@syr.edu}{hpottsru@syr.edu} (corresponding author)}

\keywords{differential graded algebra, discrete Morse theory, edge ideal, minimal free resolution, squarefree monomial ideal} 

\pgfplotsset{compat=1.18}
\begin{document}

\begin{abstract}
    Given a squarefree monomial ideal $I$ of a polynomial ring $Q$, we show that if the minimal free resolution $\mathbb{F}$ of $Q/I$ admits the structure of a differential graded (dg) algebra, then so does any ``pruning" of $\mathbb{F}$.  In the language of combinatorics, this says that if $Q/\mathcal{F}(\Delta)$, the quotient of the ambient polynomial ring by the facet ideal $\mathcal{F}(\Delta)$ of a simplicial complex $\Delta$, is minimally resolved by a dg algebra, then so is the quotient by the facet ideal of each facet-induced subcomplex of $\Delta$ (over the smaller polynomial ring).  Along with techniques from discrete Morse theory and homological algebra, this allows us to give complete classifications of the trees and cycles $G$ with $Q/I(G)$ minimally resolved by a dg algebra in terms of the length of the longest path in $G$, where $I(G)$ is the edge ideal of $G$.  
\end{abstract}

\maketitle


\section*{Acknowledgments}

We thank Trung Chau for discussions on Morse matchings and Lyubeznik graphs and the anonymous reviewers for their helpful comments.  We are also grateful for \textit{Macaulay2} \cite{M2}, which we used to compute examples of resolutions of edge ideals of diameter-four trees, leading us to conjecture and prove that these resolutions are as described in Section \ref{d4section}. Lastly, we would like to acknowledge the Fairfax Algebra Days conference that occurred in March 2024, from which this paper resulted.

\section{Introduction}

Differential graded (dg) algebras are powerful tools in homological commutative algebra.  For example, the existence of such structure reduces the difficulty in computing certain cohomologies and can be leveraged to give universal resolutions of modules, such as bar resolutions (see, for example, \cite{Iyengar,henry}).  In \cite{BooGri}, Boocher et al. use dg algebras to study deviations related to edge ideals, and dg algebras make a significant appearance in the hot-topic study of cohomological support varieties (see, for example, \cite{josh}).  Furthermore, commutative algebra is not the only area of mathematics in which one encounters dg algebras.  First arising in algebraic topology, their use in homological commutative algebra was spearheaded by Avramov, Buchsbaum, and Eisenbud, to name only a few.  


Unfortunately, the question of whether a given resolution admits the structure of a dg algebra is a difficult one.  To show existence, one often has to describe the multiplication explicitly.  The typical first examples of dg algebras one encounters are the Taylor resolution \cite{Gemeda} and the Koszul complex.  Tests for non-existence of dg algebra structure include results due to: Avramov \cite{avr}, in terms of the non-vanishing of kernels of maps between certain Tor modules; Katth\"an \cite{Kat}, in terms of $f$-vectors of simplicial complexes; and Burke \cite{Burke}, in terms of higher homotopies and $A_\infty$-structures. Notoriously, these obstructions are challenging to compute.  In this paper, we contribute to the story of dg algebra structures on resolutions in the world of squarefree monomial ideals.

\begin{main*}[Theorem \ref{pruningisdg}, Corollary \ref{iteratepruning}]
    Let $I$ be a squarefree monomial ideal of a polynomial ring $Q=\Bbbk[x_1,\ldots,x_n]$, and let $\mathbb{F}$ be the minimal $Q$-free resolution of $Q/I$.  Let $I'$ be the ideal of $R:=Q/(x_{k_1},\ldots,x_{k_\ell})$ obtained from $I$ by setting $x_{k_1}=\ldots=x_{k_\ell}=0$, i.e., $I'=I/(x_{k_1},\ldots,x_{k_\ell})$.  If $\mathbb{F}$ admits the structure of a differential graded algebra, then the minimal $R$-free resolution of $R/I'$ admits the structure of a differential graded algebra. 
\end{main*}

An active area of research within combinatorial commutative algebra is relating invariants of (finite simple) graphs to algebraic properties of their edge ideals.  For example, Fr\"oberg's Theorem \cite{Froberg} establishes an equivalence between the cochordality of a graph and the linearity of the resolution of its edge ideal.  Resolutions encode a wealth of algebraic information, such as graded Betti numbers and projective dimension.  As a consequence of Theorem \ref{pruningisdg}, we get a result relating dg algebra structure and facet ideals of facet-induced subcomplexes and thus edge ideals of induced subgraphs (see Definition \ref{facetinduced}).    
\begin{cor*}[Corollary \ref{simplex}]
     Let $\Delta$ be a simplicial complex such that the minimal $Q$-free resolution of $Q/\mathcal{F}(\Delta)$ admits the structure of a differential graded algebra, where $\mathcal{F}(\Delta)$ is the facet ideal of $\Delta$ and $Q$ is the ambient polynomial ring on the vertices of $\Delta$.  If $\Delta'$ is a facet-induced subcomplex of $\Delta$ and $Q'$ is the ambient polynomial ring of $\Delta'$, then the minimal $Q'$-free resolution of $Q'/\mathcal{F}(\Delta')$ admits the structure of a differential graded algebra.
\end{cor*}

In the case of edge ideals, Corollary \ref{simplex} may be expressed as follows.  

\begin{cor*}[Corollary \ref{main}]
     Let $G$ be a graph such that the minimal $Q$-free resolution of $Q/I(G)$ admits the structure of a differential graded algebra, where $I(G)$ is the edge ideal of $G$ and $Q$ is the ambient polynomial ring of $G$.  If $G'$ is an induced subgraph of $G$ and $Q'$ is the ambient polynomial ring of $G'$, then the minimal $Q'$-free resolution of $Q'/I(G')$ admits the structure of a differential graded algebra.
\end{cor*}

Combining Corollary \ref{main} with techniques from discrete Morse theory and homological algebra, we obtain the following classifications about the ``dg-ness" of trees and cycles based on the length of the longest path they contain.  

\begin{cla*}[Theorem \ref{classification}]
    Let $\Gamma$ be a tree of diameter $d$.  The minimal $Q$-free resolution of $Q/I(\Gamma)$ admits the structure of a differential graded algebra if and only if $d \leq 4$.
\end{cla*}

\begin{cla*}[Theorem \ref{classification2}]
    Let $C_n$ be the cycle on $n$ vertices.  The minimal $Q$-free resolution of $Q/I(C_n)$ admits the structure of a differential graded algebra if and only if $n \leq 5$.  
\end{cla*}

The outline of sections is as follows.  In Section \ref{background}, we develop the background necessary to discuss dg algebra structures on resolutions of edge ideals.  In Section \ref{Lyubezniksection}, we start the process of proving the existence portion of Theorem \ref{classification}, using tools from discrete Morse theory to show that if the minimal free resolution $\mathbb{F}$ of $Q/I(G)$ is a Lyubeznik resolution, then $\mathbb{F}$ admits the structure of a dg algebra (Theorem \ref{Lyubeznikdg}).  Trees of diameter three fall under this umbrella.  In Section \ref{d4section}, we suitably ``glue together" Taylor resolutions to construct the minimal free resolution of quotients by edge ideals of diameter-four trees and develop a dg algebra structure on the resulting resolution which respects the multiplication coming from the component Taylor resolutions (Definition \ref{defn:treeProd} and Corollary \ref{cor:treeProd}).  The non-existence side of the story is discussed in Section \ref{pruningsection}, where we describe a process due to Boocher (Definition \ref{pruningprocess}) which, for certain classes of ideals, ``prunes" the minimal free resolution $\mathbb{F}$ of $Q/I$ to the minimal free resolution $P(\mathbb{F},Z)$ of the ``pruned" quotient $Q'/I'$ \cite{boocher}.  We show that in some cases this ``pruning process" is ``dg-sensitive," i.e., the existence of a dg algebra structure on $\mathbb{F}$ implies the existence of a dg algebra structure on $P(\mathbb{F},Z)$.  Pruning combines with results of Avramov (Example \ref{avramov}) and Katth\"an (Proposition \ref{5path}) about the non-existence of a dg algebra structure on the minimal free resolution of $Q/I(P_6)$, where $P_6$ is the path on six vertices and of diameter five, to complete the classification for trees.  In Section \ref{class} we complete the story for cycles using discrete Morse theory and a theorem of Katth\"an \cite{Kat} and state the classification Theorems \ref{classification} and \ref{classification2}.  Finally, in Appendices \ref{append:GC}, \ref{append:assoc}, and \ref{append:LR}, we provide computations used to prove the existence of the dg algebra structure developed in Section \ref{d4section}.  

\section{Background}\label{background}

We begin with the necessary background to discuss dg algebra structures on resolutions, particularly those induced by Morse matchings.  We cover the Taylor and Lyubeznik resolutions and introduce discrete Morse theory, which we use in Section \ref{Lyubezniksection} to show that when the minimal free resolution of an edge ideal is a Lyubeznik resolution, it admits the structure of a dg algebra.  Throughout this paper, $Q$ is an ambient polynomial ring.  The number of variables, if relevant, will be clear from the context.  By a ``graph," we mean a finite simple undirected graph.    

\begin{defn}\label{edgeideal}
    Let $G=(V,E)$ be a graph with vertex set $V=\{x_1,\ldots,x_n\}$ and edge set $E$.  The \textit{edge ideal} $I(G)$ of $G$ is the ideal of $Q=\Bbbk[x_1,\ldots,x_n]$ generated by the squarefree quadratic monomials $x_ix_j$ such that $\{x_i,x_j\} \in E$.   
\end{defn}

\begin{ex}\label{edgeidealex}
    Let $G$ be the graph in Figure \ref{L(0,1,2)}.  The edge ideal of $G$ is the ideal $I(G)$ of $\Bbbk[x_1,x_2,x_3,x_4,x_5]$ given by $I(G)=(x_1x_2,x_1x_5,x_2x_5,x_3x_5,x_4x_5)$.  
    \begin{figure}[H]
    \centering\begin{tikzpicture}
			\draw[fill=black] (1,0) circle (3pt);
			\draw[fill=black] (2,0) circle (3pt);
			\draw[fill=black] (3,0) circle (3pt);
			\draw[fill=black] (2.5,1) circle (3pt);
                \draw[fill=black] (1.5,1) circle (3pt);

            \node[below = 2] at (1,0)  {$x_1$};
            \node[below = 2] at (2,0)  {$x_5$}; 
            \node[below = 2]  at (3,0)  {$x_3$};
            \node[above = 2] at  (1.5,1)  {$x_2$};
            \node[above = 2] at (2.5,1) {$x_4$};
			
			\draw[thick] (2,0) -- (2.5,1);
                \draw[thick] (1,0) -- (1.5,1);
                \draw[thick] (2,0) -- (1.5, 1);
			\draw[thick] (1,0) -- (2,0) -- (3,0);

	\end{tikzpicture}\caption{The graph $G=L(0,2,1)$}\label{L(0,1,2)}\end{figure}
\end{ex}

We are concerned with the minimal $Q$-free resolution of the quotient $Q/I(G)$ for certain graphs $G$.  To talk about these resolutions, we need to understand the structure of the Taylor resolution of a monomial ideal.  A common theme for combinatorially-defined resolutions of $Q/I$ is the need for a choice of ordering on the minimal generators of the ideal $I$.  By abuse of language, we will say both that $<$ is a total order on the minimal generators $u_1, \ldots,u_t$ of an ideal $I$ and that $<$ is a total order on the set $G(I)$ of minimal generators of $I$.  Implicit in Definition \ref{taylordef} is that the Taylor resolution is indeed a resolution \cite{taylor}.      

\begin{defn}[Taylor, 1966]\label{taylordef}
     Let $I$ be a monomial ideal of $Q=\Bbbk[x_1,\ldots,x_n]$ generated by the monomials $u_1, \ldots, u_t$, and let $<$ be a total order on these generators.  For $U \subseteq \{u_1,\ldots,u_t\}$, set $m_U=\lcm\{u_j \mid u_j \in U\}$.  The \textit{Taylor resolution} $\mathbb{T}$ of $Q/I$ is the $Q$-free resolution with
     \[
     \mathbb{T}_i=Q^{\binom{t}{i}},
     \]
     which has basis $e_U$, where $U$ is a subset of $\{u_1,\ldots,u_t\}$ of size $i$.  The differential of $\mathbb{T}$ is given by
     \[
     \partial(e_U)=\sum_{u \in U} (-1)^{\sigma(u,U)}\dfrac{m_U}{m_{U\backslash\{u\}}}e_{U \backslash\{u\}},
     \]
     where $\sigma(u,U)=|\{v \in U : v<u\}|$. 
     
     It is often notationally advantageous to replace $u_j$ by the index $j$ and define the Taylor resolution mutatis mutandis, using basis elements $e_U$, where $U$ is a subset of $[t]$ of size $i$.   
\end{defn} 

\begin{remark}
    It is commonplace to take $\{u_1,\ldots,u_t\}$ to be the minimal set of generators $G(I)$ of $I$, and this is the situation in which we will typically work.  
\end{remark}

Any two orderings on $G(I)$ induce isomorphic Taylor resolutions.  The chosen order will be specified in the first differential.  

\begin{ex}\label{taylorresnex}
    Let $Q=\Bbbk[x,y,z,w]$ and $I=(xw,yz,xz,xy)$.  The Taylor resolution $\mathbb{T}$ of $Q/I$ is \small
    \[
    0 \xrightarrow{} Q \xrightarrow{\left[\begin{smallmatrix}
        -1\\1\\-1\\w
    \end{smallmatrix}\right]} Q^4 \xrightarrow{\left[\begin{smallmatrix}
        1&1&0&0\\-y&0&y&0\\0&-z&-z&0\\w&0&0&1\\0&w&0&-1\\0&0&w&1
    \end{smallmatrix}\right]} Q^6 \xrightarrow{\left[\begin{smallmatrix}
        -yz&-z&-y&0&0&0\\xw&0&0&-x&-x&0\\0&w&0&y&0&-y\\0&0&w&0&z&z
    \end{smallmatrix}\right]} Q^4 \xrightarrow{\begin{bmatrix}
        xw&yz&xz&xy
    \end{bmatrix}} Q \xrightarrow{} 0
    \] \normalsize
    The chosen order, indicated by the first differential, is $xw<yz<xz<xy$.  Notice that $\mathbb{T}$ is not minimal.  
\end{ex}

Our main result (Theorem \ref{pruningisdg}) discusses the existence of a dg algebra structure, a notion we now define. 

    \begin{defn}\label{dga}
A \textit{differential graded (dg) algebra} over a ring $Q$ is a complex $(\mathcal{A},\partial^\mathcal{A})$ of free $Q$-modules equipped with a unitary, associative multiplication $\mathcal{A} \otimes_Q \mathcal{A} \to \mathcal{A}$ satisfying 
\begin{itemize}
    \item[(i)] $\mathcal{A}_i\mathcal{A}_j \subseteq \mathcal{A}_{i+j}$,
    \item[(ii)] $a_ia_j=(-1)^{ij}a_ja_i$,
    \item[(iii)] $a_i^2=0 \text{ if $i$ is odd}$,
    \item[(iv)] $\partial^\mathcal{A}(a_ia_j)=\partial^\mathcal{A}(a_i)a_j+(-1)^{i}a_i\partial^\mathcal{A}(a_j)$,
\end{itemize}
where $a_\ell \in \mathcal{A}_\ell$.  Conditions (i)-(iii) are together called \textit{graded commutativity}, and condition (iv) is called the \textit{Leibniz rule}. 
    \end{defn}  

Throughout, we will use the fact that the Taylor resolution admits the structure of a dg algebra (Example \ref{taylorex}), a celebrated result due to Gemeda \cite{Gemeda}.  However, the Taylor resolution is often far from minimal (Example \ref{taylorresnex}).      

\begin{ex}[Gemeda, 1976]\label{taylorex}
    Let $I$ be a monomial ideal of $Q=\Bbbk[x_1,\ldots,x_n]$ generated by the monomials $u_1, \ldots, u_t$, and let $<$ be a total order on these generators.  Let $U \subseteq \{u_1, \ldots, u_t\}$, and again set $m_U=\lcm\{u_j \mid u_j \in U\}$.  The Taylor resolution $\mathbb{T}$ of $Q/I$ admits the structure of a dg algebra under the product
    \[
e_V\cdot e_W=\begin{cases}
    (-1)^{\sigma(V,W)}\dfrac{m_Vm_W}{m_{V \cup W}}e_{V \cup W},& V \cap W=\emptyset,\\0,&V \cap W \neq \emptyset,
\end{cases}
\]
where $e_U$ is a basis element of $\mathbb{T}$ in degree $|U|$ and $\sigma(V,W)=|\{(v,w)\in V \times W : v>w\}|$.  
\end{ex}

Not every module's minimal free resolution admits the structure of a dg algebra.  The most famous example of this fact is due to Avramov \cite{avr}.    

\begin{ex}[Avramov, 1981]\label{avramov}
    Let $Q=\Bbbk[x,y,z,w]$ and $I=(x^2,xy,yz,zw,w^2)$.  The minimal $Q$-free resolution of $Q/I$ does not admit the structure of a dg algebra.    
\end{ex}

In \cite{Kat}, Katth\"an translates Example \ref{avramov} to the world of edge ideals.  

\begin{prop}[Katth\"an, 2019]\label{5path}
    The minimal $Q$-free resolution of $Q/I(P_6)$ does not admit the structure of a differential graded algebra, where $P_6$ is the path on six vertices.  
\end{prop}

For the sake of brevity, we make the following definition.

\begin{defn}\label{dggraph}
    Let $G$ be a graph.  If the minimal free resolution of $Q/I(G)$ admits the structure of a dg algebra, then we say that $G$ is a \textit{dg graph} or simply that $G$ (or $I(G)$) is \textit{dg}.    
\end{defn}

We are interested in relating the existence of a dg algebra structure on the minimal free resolution of $Q/I(G)$ to invariants of the graph $G$.  Theorems \ref{classification} and \ref{classification2} relate such existence for trees and cycles to the length of the longest path in $G$.
 
\begin{defn}\label{diameter}
    The \textit{diameter} of a graph $G$ is the maximum among the lengths of the shortest paths between vertices in $G$.    
\end{defn}

\begin{remark}
    For a tree $\Gamma$, the length of the longest path in $\Gamma$ and the diameter of $\Gamma$ are equivalent.
\end{remark}

The length of the longest path in a graph $G$ need not be the diameter of $G$ if $G$ is not a tree.  

\begin{ex}
    The diameter of the cycle $C_5$ on $5$ vertices is $2$, but the length of the longest path in $C_5$ is $4$.  
\end{ex}

Example \ref{taylorex} leads us to comment on the existence of dg algebra structure on the minimal free resolution of $Q/I(\Gamma)$ for trees $\Gamma$ of small diameter.  

\begin{obs}\label{d012obs}
    Let $\Gamma$ be a tree of diameter zero, one, or two. The minimal free resolution of $Q/I(\Gamma)$ is the Taylor resolution on the minimal monomial generators of $I(\Gamma)$ and thus admits the structure of a differential graded algebra. That is, trees of diameter zero, one, and two are all dg. 
\end{obs}

\begin{proof}
    Units appear in the differential of the Taylor resolution if and only if $m_U=m_{U\backslash \{u\}}$ for some subset $U$ of the minimal generators $G(I(\Gamma))$ of $I(\Gamma)$ and some $u \in U$ if and only if there exists $u \in G({I(\Gamma)})$ such that $m_{G(I(\Gamma))}=m_{G(I(\Gamma)) \backslash \{u\}}$.  When $\Gamma$ has diameter at most two, each $u \in G(I(\Gamma))$ is divisible by a unique variable, and so such an equality is impossible.  
\end{proof}

While always a dg algebra, the Taylor resolution is often far from minimal. A class of resolutions associated to monomial ideals which are both more-often minimal and typically closer to minimal than the Taylor resolution are the Lyubeznik resolutions \cite{Lyubeznik}.  Lyubeznik resolutions, even minimal ones, need not admit the structure of a dg algebra (see, for example, \cite[Corollary 5.3]{Kat}).  

\begin{defn}[Lyubeznik, 1988]\label{Lyubeznikresn}
For a total order $<$ on the minimal generators $G(I)$ of a monomial ideal $I$ of $Q$, the \textit{Lyubeznik complex} $\mathbb{L}_< \subseteq \mathbb{T}$ is generated by those basis elements $e_{\{u_{i_1}<\ldots<u_{i_s}\}}$ of $\mathbb{T}$ such that for every $1 \leq j < s$ and every $q$ with $u_q<u_{i_j}$, we have $u_q \nmid \lcm(u_{i_j},\ldots,u_{i_s})$.
\end{defn}

\begin{thm}[Lyubeznik, 1988]\label{Lyubeznikresnpf}
    The Lyubeznik complex $\mathbb{L}_<$ resolves $Q/I$ over $Q$, regardless of $<$.  We thus call $\mathbb{L}_<$ the \textit{Lyubeznik resolution} of $Q/I$ with respect to $<$.    
\end{thm}

Unlike with the Taylor resolution, distinct orderings on $G(I)$ may (or may not) induce nonisomorphic Lyubeznik resolutions.      

\begin{ex}\label{Lyubeznikresnex}
    Let $G$ be the graph in Figure \ref{L(1,1,1)}. The edge ideal $I(G)$ of $G$ is
    \[
    I(G) = (xy, xz, yz, xx_1, yy_1).
    \]
\begin{figure}[H]
    \centering\begin{tikzpicture}
			\draw[fill=black] (0,0) circle (3pt);
			\draw[fill=black] (1,0) circle (3pt);
			\draw[fill=black] (3,0) circle (3pt);
			\draw[fill=black] (4,0) circle (3pt);
                \draw[fill=black] (2,1.5) circle (3pt);

            \node[below = 2] at (0,0) {$x_1$};
            \node[below = 2] at (1,0)  {$x$}; 
            \node[below = 2]  at (3,0)  {$y$};
            \node[below = 2] at  (4,0)  {$y_1$};
            \node[above = 2] at (2,1.5) {$z$};
			
			\draw[thick] (3,0) -- (2,1.5);
                \draw[thick] (1,0) -- (2,1.5);
			\draw[thick] (0,0) -- (1,0) -- (3,0) -- (4,0);

	\end{tikzpicture}\caption{The graph $G=L(1,1,1)$}\label{L(1,1,1)}
    \end{figure}
Different total orders on the generators of $I(G)$ may induce nonisomorphic Lyubeznik resolutions over $Q = \Bbbk[x, y, x_1, y_1, z]$. Indeed, consider the following three total orders:

\begin{enumerate}
    \item The ordering $xy < xz < yz < xx_1 < yy_1$ has corresponding Lyubeznik resolution 
    \[
    0 \xrightarrow{} Q^2 \xrightarrow{\begin{bmatrix}
        x_1 & 0 \\ 0 & y_1 \\ -z & 0 \\ 0 & -z \\ y & 0 \\ 0 & x \\
    \end{bmatrix}} Q^6 \xrightarrow{\begin{bmatrix}
       -z & -z & -x_1 & -y_1 & 0 & 0 \\ y & 0 & 0 & 0 & -x_1 & 0 \\ 0 & x & 0 & 0 & 0 & -y_1 \\ 0 & 0 & y & 0 & z & 0 \\ 0 & 0 & 0 & x & 0 & z 
    \end{bmatrix}} Q^5 \xrightarrow{\begin{bmatrix}
        xy & xz & yz& xx_1 & yy_1
    \end{bmatrix}} Q \xrightarrow{} 0.
    \]

    \item The ordering $xz < yz < yy_1 < xy < xx_1$ has corresponding Lyubeznik resolution 
    \[
    0 \xrightarrow{} Q \xrightarrow{\begin{bmatrix}
        0 \\ -x_1 \\ 1 \\ -y_1 \\ z
    \end{bmatrix}} Q^5 \xrightarrow{\partial_2} Q^8 \xrightarrow{\partial_1} Q^5 \xrightarrow{\begin{bmatrix}
      xz & yz & yy_1 &  xy & xx_1 
    \end{bmatrix}} Q \xrightarrow{} 0,
    \]
    where
    \[
    \partial_1=\begin{bmatrix}
        -y & -yy_1 & -y & -x_1 & 0 & 0 & 0 & 0\\ x & 0 & 0 & 0 & -y_1 & 0 & 0 & 0\\ 0 & xz & 0 & 0 & z & -x & -xx_1 & 0\\ 0 & 0 & z & 0 & 0 & y_1 & 0 & -x_1\\ 0 & 0 & 0 & z & 0 & 0 & yy_1 & y
    \end{bmatrix}\,\text{ and }\,
    \partial_2=\begin{bmatrix}
        y_1 & 0 & 0 & 0 & 0\\ -1 & 1 & x_1 & 0 & 0\\ 0 & -y_1 & 0 & x_1 & 0\\ 0 & 0 & -yy_1 & -y & 0\\ x & 0 & 0 & 0 & 0\\ 0 & z & 0 & 0 & x_1\\ 0 & 0 & z & 0 & -1\\ 0 & 0 & 0 & z & y_1
    \end{bmatrix}.
    \]

    \item The ordering $yz < xz < yy_1 < xy < xx_1$ has corresponding Lyubeznik resolution 
    \[
    0 \xrightarrow{} Q \xrightarrow{\begin{bmatrix}
        0 \\ -x_1 \\ 1 \\ -y_1 \\ z
    \end{bmatrix}} Q^5 \xrightarrow{\partial_2} Q^8 \xrightarrow{\partial_1} Q^5 \xrightarrow{\begin{bmatrix}
      yz & xz & yy_1 &  xy & xx_1 
    \end{bmatrix}} Q \xrightarrow{} 0,
    \]
    where
    \[
    \partial_1=\begin{bmatrix}
        -x & -y_1 & -x & -xx_1 & 0 & 0 & 0 & 0\\ y & 0 & 0 & 0 & -x_1 & 0 & 0 & 0\\ 0 & z & 0 & 0 & 0 & -x & -xx_1 & 0\\ 0 & 0 & z & 0 & 0 & y_1 & 0 & -x_1\\ 0 & 0 & 0 & yz & z & 0 & yy_1 & y
    \end{bmatrix}\,\text{ and }\,\partial_2=\begin{bmatrix}
        x_1 & 0 & 0 & 0 & 0\\ 0 & x & xx_1 & 0 & 0\\ 0 & -y_1 & 0 & x_1 & 0\\ -1 & 0 & -y_1 & -1 & 0\\ y & 0 & 0 & 0 & 0\\ 0 & z & 0 & 0 & x_1\\ 0 & 0 & z & 0 & -1\\ 0 & 0 & 0 & z & y_1
    \end{bmatrix}.
    \]
  
\end{enumerate}

Resolutions (2) and (3) are isomorphic, and they are clearly not isomorphic to resolution (1), the minimal free resolution of $Q/I(G)$. 
\end{ex} 

In Section \ref{Lyubezniksection}, we investigate the graphs $G$ for which $Q/I(G)$ is minimally resolved by a Lyubeznik resolution.  In general, we have the following terminology. 

\begin{defn}\label{Lyubenikideal}
    If there exists a total order $<$ on $G(I)$ such that $\mathbb{L}_<$ is the minimal free resolution of $Q/I$, then we call $I$ a \textit{Lyubeznik ideal} of $Q$.   
\end{defn}

\begin{ex}\label{Lyubeznikresnexcontd} 
    The ideal $I(G)=(xy, xz, yz, xx_1, yy_1) \subset \Bbbk[x,y,x_1,y_1,z]$ is a Lyubeznik ideal, since $\mathbb{L}_<$ is minimal for the ordering $xy<xz<yz<xx_1<yy_1$ (see resolution (1) from Example \ref{Lyubeznikresnex}).
\end{ex}

We work with an equivalent definition of Lyubeznik resolutions which appears in the context of discrete Morse theory, an area of combinatorial algebra to which we now turn our attention.  In Section \ref{Lyubezniksection}, we use tools from discrete Morse theory to show that Lyubeznik graphs (i.e., graphs with Lyubeznik edge ideals), and thus trees of diameter three, are dg.  

\begin{defn}\label{Taylorgraph}
    Let $I$ be a monomial ideal of $Q$.  Any free resolution of $Q/I$ induces a poset which keeps track of the nonzero components of the differential.  That is, the elements of this poset correspond to the basis elements of the free modules in the resolution, and two elements corresponding to basis elements in adjacent homological degrees are comparable if one appears with nonzero coefficient in the image of the other under the differential.  The \textit{Taylor poset} of $I$ is the poset induced in this way from $\mathbb{T} \otimes \Bbbk$, excluding the elements corresponding to homological degrees zero and one.
\end{defn}

We arrange the elements of the Taylor poset in columns: in the $i$th column, the elements are the subsets of $G(I)$ of size $i$, listed vertically.  Arrows are between elements as follows: for $i \geq 2$, direct element $\sigma$ of size $i$ to element $\tau$ of size $i-1$ if and only if $\sigma \supseteq \tau$, $m_\sigma=m_\tau$, and $|\tau| \geq 2$.

\begin{ex}\label{Taylorgraphex}
    Let $G$ be the graph in Figure \ref{L(1,1,1)}.  The edge ideal $I(G)$ of $G$ is $(xy, xz, yz, xx_1, yy_1)$.  Figure \ref{tg} shows the Taylor poset of $I(G)$.
    \begin{figure}[H]
        \centering
    \begin{tikzcd}
        & & {\{xy, xz, yz, xx_1\}} \arrow[rr] \arrow[rrd] \arrow[rrdd] & & {\{xy,yz,xx_1\}} \arrow[rrd] & & {\{xy,xx_1\}}\\
        & &     & & {\{xy,xz,xx_1\}} & & {\{yz,xx_1\}} \\
        & & {\{xy,yz,xx_1,yy_1\}} \arrow[rrd] & & {\{xz,yz,xx_1\}} \arrow[rru] & & {\{xz,xx_1\}} \\
        & &     & & {\{yz,xx_1,yy_1\}} & & {\{xy,xz\}}\\
        {\{xy,xz,yz,xx_1,yy_1\}} \arrow[rr] & & {\{xz,yz,xx_1,yy_1\}} \arrow[rru] \arrow[rrd] & & {\{xy,xz,yz\}} \arrow[rru] \arrow[rrd] \arrow[rr] & & {\{xz,yz\}}\\
        & &     & & {\{xz,xx_1,yy_1\}} & & {\{xy,yz\}}\\
        & & {\{xy,xz,xx_1,yy_1\}} \arrow[rru] & & {\{xz,yz,yy_1\}} \arrow[rrd] & & {\{xy,yy_1\}}\\
        & &     & & {\{xy,yz,yy_1\}} & & {\{xz,yy_1\}}\\
        & & {\{xy,xz,yz,yy_1\}} \arrow[rruu] \arrow[rru] \arrow[rr] & & {\{xy,xz,yy_1\}} \arrow[rru] & & {\{yz,yy_1\}}\\
        & &     & & {\{xy,xx_1,yy_1\}} \arrow[rr] & & {\{xx_1,yy_1\}}\\
    \end{tikzcd}
        \caption{Taylor poset of $I(G)=(xy, xz, yz, xx_1, yy_1)$}
        \label{tg}
    \end{figure}
\end{ex}

\begin{defn}\label{Morsematching}
    View the Taylor poset as a directed graph.  A collection of (directed) edges $A=\{(a_j,b_j)\}_{j \in \Lambda}$ of the Taylor poset is a \textit{Morse matching} if
\begin{itemize}
    \item[(i)] no two edges of $A$ are incident,
    \item[(ii)] the graph obtained by reversing the direction of the edges in $A$ is acyclic (as a directed graph).
\end{itemize}
We write $A_+$ for the collection of sources of edges in $A$ and $A_-$ for the collection of targets of edges in $A$, i.e., if $A=\{(a_j,b_j)\}_{j \in \Lambda}$, then $A_+=\{a_j \mid j \in \Lambda\}$ and $A_-=\{b_j \mid j \in \Lambda\}$.  
\end{defn}

We examine Morse matchings in Section \ref{Lyubezniksection} to deduce results about the existence of dg algebra structure on resolutions of $Q/I(G)$, where $I(G)$ is Lyubeznik. We introduce the following to aid our analysis.

\begin{defn}
    Given an edge $(a_j,b_j)$ in a Morse matching on a Taylor poset, if $b_j=a_j \backslash\{c\}$, then we say that $c$ \textit{drops from} $a_j$ and that $a_j$ \textit{drops} $c$.  
\end{defn}

\begin{ex}\label{Morsematchingsubcomplexex}
    Let $G$ be the graph in Figure \ref{L(1,1,1)}. The edge ideal $I(G)$ of $G$ is $(xy, xz, yz, xx_1, yy_1)$. Figure \ref{mmdiag} shows a Morse matching $A$ on the Taylor poset of $Q/I(G)$ (Figure \ref{tg}), where each arrow is labeled by the element dropped from its source and arrows not in $A$ are omitted.  
    \begin{figure}[H]
        \centering
 \begin{tikzcd}
        & & {\{xy, xz, yz, xx_1\}}  \arrow[rrdd, "xy"] & & {\{xy,yz,xx_1\}} \arrow[rrd, "xy"] & & {\{xy,xx_1\}}\\
        & &     & & {\{xy,xz,xx_1\}} & & {\{yz,xx_1\}} \\
        & & {\{xy,yz,xx_1,yy_1\}} \arrow[rrd, "xy"] & & {\{xz,yz,xx_1\}}  & & {\{xz,xx_1\}} \\
        & &     & & {\{yz,xx_1,yy_1\}} & & {\{xy,xz\}}\\
        {\{xy,xz,yz,xx_1,yy_1\}} \arrow[rr, "xy"] & & {\{xz,yz,xx_1,yy_1\}} & & {\{xy,xz,yz\}}  \arrow[rr, "xy"]  & & {\{xz,yz\}}\\
        & &     & & {\{xz,xx_1,yy_1\}} & & {\{xy,yz\}}\\
        & & {\{xy,xz,xx_1,yy_1\}} \arrow[rru, "xy"] & & {\{xz,yz,yy_1\}}  & & {\{xy,yy_1\}}\\
        & &     & & {\{xy,yz,yy_1\}} & & {\{xz,yy_1\}}\\
        & & {\{xy,xz,yz,yy_1\}} \arrow[rruu, "xy"] & & {\{xy,xz,yy_1\}} \arrow[rru, "xy"] & & {\{yz,yy_1\}}\\
        & &     & & {\{xy,xx_1,yy_1\}} \arrow[rr, "xy"] & & {\{xx_1,yy_1\}}\\
    \end{tikzcd}
 \caption{A Morse matching}
        \label{mmdiag}
\end{figure}
\end{ex}

Morse matchings allow us to ``cut down" from the Taylor resolution to complexes which are closer to minimal (see \cite[Proposition 2.2, Proposition 3.1, Lemma 7.7]{BW} and also \cite[Theorem 2.3]{BM}, the statement of which we reformulate to match our setting).    

\begin{prop}[Batzies-Welker, 2002]\label{Morsematchingsubcomplex}
    A Morse matching $A$ on the Taylor poset of a monomial ideal $I$ induces a subcomplex $\mathbb{J}$ of the Taylor resolution $\mathbb{T}$ on $G(I)$ corresponding to (some of) the nonminimal part of $\mathbb{T}$:
    \[
\mathbb{J}:=\bigoplus_{V \in A_+} 0 \to Qe_V \to Q\partial(e_V) \to 0.
    \]  
    Furthermore, $\mathbb{J}$ is exact, and thus the induced complex $\mathbb{T}/\mathbb{J}$ is a resolution of $Q/I$. 
\end{prop}

Let $f_i$ be the rank of the $i$th free module in the Taylor resolution of $Q/I(G)$, and let $s_i$ be the number of $i$-element subsets of $G(I)$ in $A_+\cup A_-$.  If $\mathbb{J}$ is as in Proposition \ref{Morsematchingsubcomplex}, then the rank of the $i$th free module in $\mathbb{T}/\mathbb{J}$ is $f_i-s_i$, since we are either sending a basis element of $\mathbb{T}$ to zero (in the case of $A_+$) or performing a change of basis (in the case of $A_-$) upon quotienting by $\mathbb{J}$.

\begin{ex}\label{mmrmk}
    In the setting of Example \ref{Morsematchingsubcomplexex}, upon taking the quotient $\mathbb{T}/\mathbb{J}$, where $\mathbb{J}$ is the subcomplex induced by $A$ (see Proposition \ref{Morsematchingsubcomplex}), we obtain the minimal free resolution of $Q/I(G)$ (resolution (1) in Example \ref{Lyubeznikresnex}).  One way to see this is to consider Betti numbers.  In Example \ref{Lyubeznikresnex}, we saw that the Betti numbers $\beta^Q_i$ of $Q/I(G)$ are 1, 5, 6, and 2.  Note that $\beta^Q_i=f_i-s_i$ for each $i$, and so $\mathbb{T}/\mathbb{J}$ is minimal.  For example, when $i=2$, we have that $6=10-4$. 
\end{ex}

Quotienting the Taylor resolution by a subcomplex induced by a Morse matching yields a resolution \cite{BW}, and it may even yield a minimal one (see Example \ref{mmrmk}).  A sufficiently nice Morse matching $A$ cuts down the Taylor resolution $\mathbb{T}$ in such a way as to preserve the dg algebra structure on $\mathbb{T}$.  This occurs when the subcomplex induced by $A$ is a dg ideal of $\mathbb{T}$.

\begin{defn}\label{dgidealdef}
    A subcomplex $\mathbb{J}$ of a dg algebra $\mathbb{F}$ is a \textit{dg ideal} of $\mathbb{F}$ if $\mathbb{F}\mathbb{J} \subseteq \mathbb{J}$, i.e., if $\mathbb{J}$ is closed under multiplication by $\mathbb{F}$.    
\end{defn}

When $\mathbb{J}$ is a dg ideal of a dg algebra $\mathbb{F}$, the quotient $\mathbb{F}/\mathbb{J}$ is again a dg algebra.  We now give a sufficient condition on a Morse matching $A$ for the subcomplex induced by $A$ to be a dg ideal of the Taylor resolution.  

\begin{remark}\label{timpaper}
    Developed separately, the proof of \cite[Theorem 4.2]{tim} describes essentially the same condition that we do in the following theorem.  
\end{remark}

\begin{thm}\label{supersetdg}
    Let $A$ be a Morse matching on the Taylor poset of an ideal $I$ of $Q$.  If $A_+$ is closed under taking supersets (i.e., whenever $a_j \in A_+$, we have $a_k \in A_+$ for all $a_k\supseteq a_j$), then the subcomplex $\mathbb{J}$ induced by $A$ is a dg ideal of the Taylor resolution $\mathbb{T}$ of $Q/I$.  Thus, the resulting quotient $\mathbb{T}/\mathbb{J}$ admits the structure of a differential graded algebra.    
\end{thm}

\begin{proof}
    Suppose $(U,U') \in A$, and let $e_V \in \mathbb{T}$ be a basis element.  Since $A_+$ is closed under taking supersets, the multiplication on $\mathbb{T}$ (Example \ref{taylorex}) shows that each basis element of $\mathbb{T}$ indexed by a superset of $U$ is in $A_+$, and so $e_Ve_U \in \mathbb{J}$.  By the Leibniz rule,
    \[
    \partial(e_Ve_U)=\partial(e_V)e_U+(-1)^{|V|}e_V\partial(e_U).
    \]
    The same logic as above yields that $\partial(e_V)e_U \in \mathbb{J}$.  Since $e_Ve_U \in \mathbb{J}$, we have that $\partial(e_Ve_U) \in \mathbb{J}$, as well.  It follows that $e_V\partial(e_U) \in \mathbb{J}$, and so we have that $\mathbb{J}$ is a dg ideal of $\mathbb{T}$. 
\end{proof}

\begin{remark}
The condition that $A_+$ is closed under taking supersets is equivalent to saying that $A_+$ generates an upper order ideal of the boolean poset on $G(I)$.      
\end{remark}

In particular, if a Morse matching $A$ induces the minimal free resolution of $Q/I$ and $A_+$ is closed under taking supersets, then Theorem \ref{supersetdg} says that the minimal free resolution of $Q/I$ admits the structure of a dg algebra.  This is the main tool used in Section \ref{Lyubezniksection}.  

To conclude Section \ref{background}, we provide an example to show that the condition on $A_+$ from Theorem \ref{supersetdg}, while sufficient, is not necessary.  

\begin{ex}\label{notiff}
Theorem \ref{supersetdg} is not a biconditional statement.  Consider the cycle on five vertices $C_5$ (Figure \ref{C5}), which has edge ideal 
\[I(C_5)=(xy,yz,zu,uv,xv).\]
\begin{figure}[H]
    \centering\begin{tikzpicture}
        \draw[fill=black] (0,1) circle (3pt);
        \draw[fill=black] (2,1) circle (3pt);
        \draw[fill=black] (1,1.75) circle (3pt);
        \draw[fill=black] (0.5,0) circle (3pt);
        \draw[fill=black] (1.5,0) circle (3pt);

        \node[below = 2] at (0.5,0) {$u$};
        \node[below = 2] at (1.5,0) {$z$};
        \node[above = 2] at (1,1.75) {$x$};
        \node[right = 2] at (2,1) {$y$};
        \node[left = 2] at (0,1) {$v$};

        \draw[thick] (0.5,0) -- (1.5,0);
        \draw[thick] (0.5,0) -- (0,1);
        \draw[thick] (0,1) -- (1,1.75);
        \draw[thick] (1,1.75) -- (2,1);
        \draw[thick] (2,1) -- (1.5,0);
    \end{tikzpicture}\caption{Cycle on five vertices $C_5$}\label{C5}
    \end{figure}

    By  \cite[Remark 4.24]{BM}, the minimal free resolution of $Q/I(C_5)$ is induced by a Morse matching. When constructing a Morse matching $A$ (Figure \ref{MMC5}) to induce the minimal free resolution of $Q/I(C_5)$, we necessarily include the edges $$\{xy,yz,zu\} \rightarrow \{xy,zu\} \,\text{ and }\, \{yz,zu,uv\} \rightarrow \{yz,uv\}$$ in $A$. To meet the condition in Theorem \ref{supersetdg}, we would need to have $\{xy,yz,zu,uv\}$ in $A_+$, which it is not.  For the subcomplex $\mathbb{J}$ induced by $A$ to be a dg ideal, we need the weaker condition that $e_{\{xy,yz,zu,uv\}}$ is in $\mathbb{J}$. 

    Examining the differential of $e_{\{xy,yz,zu,uv,xv\}}$, we get that
    \begin{equation*}
   e_{\{xy,yz,zu,uv\}} = \partial(e_{\{xy,yz,zu,uv,xv\}}) - e_{\{xy,zu,uv,xv\}}-e_{\{xy,yz,zu,xv\}} + e_{\{yz,zu,uv,xv\}} + e_{\{xy,yz,uv,xv\}}
    \end{equation*}
    is in the subcomplex $\mathbb{J}$ induced by $A$, and so $\mathbb{J}$ is a dg ideal, even though $\{xy,yz,zu,uv\}$ is not in $A_{+}$.  

    \begin{figure}[H]
       \centering
       \begin{tikzcd}
           & & {\{xy,yz,zu,xv\}} \arrow[rr, "yz"] & & {\{xy,zu,xv\}} & & {\{xy,xv\}}\\
           & &      & & {\{xy,yz,zu\}} \arrow[rr, "yz"] & & {\{xy,zu\}}\\
            & & {\{xy,yz,uv,xv\}} \arrow[rr, "xv"] & & {\{xy,yz,uv\}} & & {\{xy,yz\}}\\
            & &      & & {\{xy,yz,xv\}} \arrow[rr, "xy"] & & {\{yz,xv\}}\\
            {\{xy,yz,zu,uv,xv\}} \arrow[rr, "xv"] & & {\{xy,yz,zu,uv\}} & & {\{zu,uv,xv\}} \arrow[rr, "uv"] & & {\{zu,xv\}}\\
            & &      & & {\{xy,uv,xv\}} \arrow[rr, "xv"] & & {\{xy,uv\}}\\
            & & {\{xy,zu,uv,xv\}} \arrow[rr, "xv"] & & {\{xy,zu,uv\}} & & {\{zu,uv\}}\\
          & &      & & {\{yz,zu,uv\}} \arrow[rr, "zu"] & & {\{yz,uv\}}\\
          & & {\{yz,zu,uv,xv\}} \arrow[rr, "uv"] & & {\{yz,zu,xv\}}  & & {\{yz,zu\}}\\
          & &   & &   {\{yz,uv,xv\}} & & {\{uv,xv\}}\\
       \end{tikzcd}\caption{A Morse matching on the Taylor poset of $I(C_5)$}\label{MMC5}
    \end{figure}  
\end{ex}

\section{Lyubeznik Edge Ideals}\label{Lyubezniksection}

In this section, we show that Lyubeznik edge ideals are dg.  The class of Lyubeznik graphs, i.e., graphs with Lyubeznik edge ideal, contains the class of trees of diameter three, and so this section provides part of the classification of dg trees (Theorem \ref{classification}).  In \cite{chm}, graphs with Lyubeznik edge ideals are characterized as the graphs $L(a,b,c)$ with vertex set
\[
V=\{x,y\}\cup\{x_i\}_{i=1}^a\cup\{y_j\}_{j=1}^b\cup\{z_k\}_{k=1}^c
\]
and edge set
\[
E=\{\{x,y\}\} \cup\{\{x,x_i\}\}_{i=1}^a \cup \{\{y,y_j\}\}_{j=1}^b\cup\{\{x,z_k\},\{y,z_k\}\}_{k=1}^c.  
\]
Note that any combination of $a$, $b$, and $c$ may be taken to be zero.  

\begin{ex}\label{L(a,b,c)}
    Figures \ref{L(0,1,2)} and \ref{L(1,1,1)} show the Lyubeznik graphs $L(0,2,1)$ and $L(1,1,1)$, respectively.  
\end{ex}

Equivalent to Definition \ref{Lyubeznikresn}, the Lyubeznik resolutions of a monomial ideal $I=(u_1,\ldots,u_t)$ may be described as follows (see \cite[Theorem 3.2]{BW} and also \cite[Proposition 2.1]{chm}, the statement of which we reformulate to match our setting). 

\begin{thm}[Batzies-Welker, 2002]\label{LyubeznikMorsematching}   
Let $<$ be a total order on $G(I)=\{u_1,\ldots,u_t\}$.  Without loss of generality, say $u_1<\cdots <u_t$. For $\sigma \subseteq G(I)$, set
    \[
    M_<(\sigma) = \min \{  u_q \in G(I) \colon u_q \textup{ divides } \lcm(\{u_{q+1},\dots, u_t\}\cap \sigma) \}.
    \]
    Then the set
    \[
    A(<):=\{  (\sigma \cup M_<(\sigma),\sigma \setminus M_<(\sigma)) \colon \sigma\subseteq G(I) \textup{ such that } M_<(\sigma) \neq \infty \}.
    \]
    is a Morse matching inducing the Lyubeznik resolution $\mathbb{L}_<$ of $Q/I$.  
\end{thm}

\begin{ex}\label{LuybeznikMorsematchingex}
    Let $G=L(1,1,1)$ (Figure \ref{L(1,1,1)}).  Figure \ref{mmdiag} shows $A(<)$ for the ordering
    \[
    xy < xz < yz < xx_1 < yy_1
    \]
    on the minimal generators of $I(G)$.  The Morse matching $A(<)$ induces the minimal free resolution of $Q/I(G)$ (Example \ref{Lyubeznikresnex}, resolution (1)).  
\end{ex}

As we have noted, different total orders need not induce the same Lyubeznik resolution of $Q/I$.  For the graphs $L(a,b,c)$, the orders $<$ such that $\mathbb{L}_<$ is the minimal $Q$-free resolution of $Q/I(L(a,b,c))$ are found in the proof of \cite[Proposition 3.6]{chm}.

\begin{thm}[Chau-H\'a-Maithani, 2024]\label{totalorder}
    For total orders $<$ on $G(I(L(a,b,c)))$ in which $xy$ is taken least, the induced Lyubeznik resolution $\mathbb{L}_<$ is the minimal $Q$-free resolution of $Q/I(L(a,b,c))$.  
\end{thm}

We track the drops in a Morse matching for a Lyubeznik resolution to reveal the Betti numbers of $Q/I(L(a,b,c))$.

\begin{corollary}\label{LyubeznikBetti}
    The $i$th Betti number of $Q/I(L(a,b,c))$ is
    \[
    \beta^Q_i=\begin{cases}
        1,&i=0,\\\\
        a+b+2c+1,&i=1,\\\\
        \displaystyle{\binom{a+c+1}{i}}+{\binom{b+c+1}{i}},&i \geq 2.
    \end{cases}
    \]
\end{corollary}

\begin{proof}
    We count the number of times an element drops from a source in the column of the Taylor poset corresponding to homological degree $i$, $i \geq 2$, in the Morse matching $A(<)$ from Theorem \ref{LyubeznikMorsematching}.

    Note that monomials corresponding to leaves cannot drop, since each is divisible by a unique variable.  Without loss of generality, if $xz_k$ drops from $\sigma$, then $x\beta, yz_k \in \sigma$, where $\beta$ is a variable such that $x\beta \in G(I(L(a,b,c)))$, whence $xy$ divides $m_\sigma$.  That is, the only element ever dropped in $A(<)$ is $xy$.  Moreover, whenever $xy$ has the potential to drop, it does.  Define
    \[
    K_i:={\binom{a+b+2c}{i-1}}-{\binom{a+c}{i-1}}-{\binom{b+c}{i-1}}.
    \]
    This value tells us how many $i$-element subsets of $G(I(L(a,b,c)))$ drop $xy$: there are ${\binom{a+b+2c}{i-1}}$ subsets containing $xy$, and in order for one of them, say $S$, to drop $xy$, we cannot take the other $i-1$ elements in $S$ to be either all $x\beta_1$ (respectively, $y\beta_2$), since then $y$ (respectively, $x$) would not divide the lcm of the elements in $S\backslash\{xy\}$.  Thus, we subtract off these situations.  

    The $i$th Betti number of $Q/I(L(a,b,c))$ is thus equal to the number of $i$-element subsets of $G(I(L(a,b,c)))$ minus the number of $i$-element subsets of $G(I(L(a,b,c)))$ that drop $xy$ minus the number of $(i+1)$-element subsets of $G(I(L(a,b,c)))$ that drop $xy$, since both $A(<)_+$ and $A(<)_-$ contribute to the ranks of the free modules in the induced resolution (see the paragraph before Example \ref{mmrmk}).  That is,
    \[
    \beta^Q_i={\binom{a+b+2c+1}{i}}-(K_i+K_{i+1}).
    \]
    By Pascal's Identity,
    \[
    \beta^Q_i={\binom{a+c+1}{i}}+{\binom{b+c+1}{i}}
    \]
    for $i \geq 2$. The cases $i=0$ and $i=1$ are clear.    
\end{proof}

As an immediate consequence, we get the projective dimension of $Q/I(L(a,b,c))$.  

\begin{corollary}\label{Lyubeznikpd}
    The projective dimension of $Q/I(L(a,b,c))$ over $Q$ is the maximum of $a+c+1$ and $b+c+1$.  
\end{corollary}

We note that the minimal free resolution of $Q/I(L(a,b,c))$ is linear, and so Corollary \ref{LyubeznikBetti} provides not only the Betti numbers of $Q/I(L(a,b,c))$, but the graded Betti numbers.  

\begin{remark}\label{gddbettislyubeznik}
    Lyubeznik graphs are examples of \textit{split} graphs, i.e., graphs whose vertices may be partitioned into a clique and an independent set (take either $\{x,y\}$ if there are no cycles or a triangle if one exists as the clique; the rest of the vertices form an independent set).  Split graphs are cochordal (see, for example \cite[Proposition 3.10]{split}), and cochordal graphs are precisely the graphs with edge ideal having a linear resolution \cite{Froberg}.  
\end{remark}

By analyzing the Morse matching taken to produce the minimal free resolution $\mathbb{L}_<$ of $Q/I(L(a,b,c))$, we deduce that $\mathbb{L}_<$ admits the structure of a dg algebra.  

\begin{thm}\label{Lyubeznikdg}
    Lyubeznik edge ideals are dg.
\end{thm}

\begin{proof}
The only element dropped in the Morse matching $A(<)$, where $<$ is an order in which $xy$ is taken least among all minimal generators of $G(I(L(a,b,c)))$, is $xy$. Thus, if $xy$ drops from $\sigma$, then $xy$ drops from $\tau \supseteq \sigma$.  That is, $A(<)_+$ is closed under taking supersets.  The matching $A(<)$ induces the minimal free resolution of $Q/I(L(a,b,c))$ by Theorem \ref{totalorder}, and by Theorem \ref{supersetdg}, $I(L(a,b,c))$ is dg.      
\end{proof}

The class of Lyubeznik graphs contains the class of diameter-three trees, and so we get the following corollary.  

\begin{corollary}\label{d3dg}
    Let $\Gamma$ be a tree of diameter three.  The minimal free resolution of $Q/I(\Gamma)$ admits the structure of a differential graded algebra.  
\end{corollary}

\begin{proof}
    Diameter-three trees are precisely the Lyubeznik graphs $L(a,b,0)$ with $a,b \geq 1$.
\end{proof}

\begin{remark}\label{powersrmk}
    The Betti numbers and projective dimension of powers of edge ideals of trees of diameter three are computed in \cite{rachel}, where they are called ``whisker graphs."  The proof uses different methods, as the story there is about linear quotients.  In \cite{faridietal}, edge ideals of trees of diameter three are shown to be minimally resolved by their Scarf complexes.  
\end{remark}

\section{Diameter-four Trees}\label{d4section}

Edge ideals of trees of diameter four are not Lyubeznik, and so the results of Section \ref{Lyubezniksection} are not applicable to them.  In this section, we construct the minimal free resolution of $Q/I(\Gamma)$, where $\Gamma$ is a tree of diameter four, from Taylor resolutions of the edge ideals of certain subgraphs of $\Gamma$.  We explicitly describe a product on the resolution that respects the multiplications coming from the constituent Taylor resolutions. For brevity, if $\FF$ is a resolution of $Q/I$, we write $\FF \simeq Q/I$.

Let $\Gamma$ be a tree of diameter $d = 4$. By \cite{d4}, we know that there exist $n, a_1, a_2, \ldots, a_n \in \ZZ$ such that the vertices of $\Gamma$ can be labeled \[z, x_1, \ldots, x_n, y_{1,1}, \ldots, y_{1,a_1}, y_{2,1}, \ldots, y_{2,a_2}, \ldots, y_{n,1}, \ldots, y_{n,a_n} \] so that $\Gamma$ has edge ideal 
\begin{align*}
I(\Gamma) = \underset{=:I}{\underbrace{(zx_1, \ldots, zx_n)}} + \underset{=:J}{\underbrace{(x_1 y_{1,1}, \ldots x_1y_{1,a_1}, x_2y_{2,1}, \ldots, x_ny_{n,a_n})}}. \label{Iform} \tag{$1$}
\end{align*}
In this section, we construct a free resolution $Q/I(\Gamma)$ from free resolutions of $Q/I$ and $Q/J$. Furthermore, we show that using minimal free resolutions in our construction produces another minimal resolution. Moving forward, let $\FF$ and $\GG$ be Taylor resolutions over $Q$ such that $\FF \simeq Q/I$ and $\GG \simeq Q/J$.

\begin{lem}
    Both $\FF$ and $\GG$ are minimal resolutions.
\end{lem}

\begin{proof}
    This follows from Observation \ref{d012obs} since $I$ is the edge ideal of a tree of diameter 2 and $J$ is the edge ideal of a forest consisting of trees of diameter at most 2.
\end{proof}

If $IJ = I \cap J$, then we would have $\FF \otimes_Q \GG \simeq Q/(I+J) = Q/ I(\Gamma)$. However, if $J \neq 0$, then there exist $i,j$ such that $x_i y_{i,j} \in J$ and thus $zx_i y_{i,j} \in (I \cap J) \backslash IJ$. Consequently, $\FF \otimes_Q \GG$ has nontrivial homology in positive homological degrees; that is, $\FF \otimes_Q \GG$ is not a resolution. This necessitates a new approach.

\begin{appr}\label{rem:SES}
    Our alternative construction leverages the exact sequence of $Q$-modules
    \[
    \xymatrix{
    0 \ar[r] & \frac{I(\Gamma)}{I} \ar[r] & \frac{Q}{I} \ar[r] & \frac{Q}{I(\Gamma)} \ar[r] & 0.
    }
    \]
    Our strategy is to construct a mapping cone such that the corresponding long exact sequence in homology reduces to the above exact sequence. To do this, we need to construct the minimal free resolution $\GG' \simeq I(\Gamma) / I$ and a chain map $\Psi: \GG' \to \FF$ such that $\Cone{\Psi} \simeq Q/I(\Gamma)$.
\end{appr}

As a first step, we construct $\GG'$. We use the following lemmas to motivate how we will manipulate the resolution $\GG$ to obtain $\GG'$.

\begin{lem}\label{lem:zColon}
Let $I$ and $J$ be as in Equation \ref{Iform}. We have $(I:_Q J) = (z)$.
\end{lem}

\begin{proof}
By the definitions of $I(\Gamma)$, $I$, and $J$, we have that if $x_i y_{i,j} \in J$, then $z x_i \in I$. We observe the following, where the first equality results from the fact that we are working with monomial ideals:
\begin{align*}
    (I:_Q x_i y_{i,j} ) &= (zx_i :_Q x_i y_{i,j}) + \sum_{\ell \neq i} (z x _{\ell} :_Q x_i y_{i,j}) \\
     &= (z) + \sum_{\ell \neq i} (z x_{\ell}) \\
     &= (z).
\end{align*}
From here, we obtain our result by noting the following equalities:
\[
(I :_Q J) = \sum_{x_i y_{i,j} \in J} (I :_Q x_i y_{i,j}) = \sum_{x_i y_{i,j} \in J} (z) = (z). \qedhere
\]
\end{proof}

We use Lemma \ref{lem:zColon} to find a $Q$-module isomorphic to $I(\Gamma)/I$. This alternate choice of $Q$-module illuminates the connection between $\GG$ and $\GG'$.

\begin{lem}\label{lem:JzJ}
    Let $I(\Gamma)$, $I$, and $J$ be as in Equation \ref{Iform}.  Then, $I(\Gamma) / I \cong J/zJ$.
\end{lem}

\begin{proof}
    For the moment, suppose $I \cap J = (I:_Q J) J$. This then yields the following:
    \[
    \frac{I(\Gamma)}{I} = \frac{I+J}{I} \cong \frac{J}{I\cap J} = \frac{J}{(I:_Q J) J} = \frac{J}{zJ}.
    \]
    Thus, for the result to hold, we must show $I \cap J = (I:_Q J) J$. By definition of $(I:_Q J)$, we must have $(I :_Q J)J \subseteq I \cap J$. To get the other containment, we note that $I \subset (z)$ and $(z) \cap J = zJ$.  It follows that \[I \cap J \subseteq (z) \cap J = zJ = (I :_Q J)J.\] 
    Thus, $(I:_Q J) J = I \cap J$, and so the desired isomorphism holds.
\end{proof}

The isomorphism $I(\Gamma) / I \cong J / zJ$ suggests another instance of a mapping cone. Specifically, since $\GG \simeq Q/J$, one has $\Sigma^{-1} \GG_{> 0} \simeq J$. Moreover, since $\GG$ is the minimal free resolution of $Q/J$, then $\Sigma^{-1} \GG_{>0}$ is the minimal free resolution of $J$.

\begin{lem}\label{lem:gPrime}
    Let $\mu^z: \Sigma^{-1} \GG_{>0} \to \Sigma^{-1} \GG_{>0}$ be the chain map given by $\mu^z(g) := zg$. The complex $\Cone{\mu^z}$ is the minimal free resolution of $J/zJ$ over $Q$; that is, $\GG' \cong \Cone{\mu^z}$.  
\end{lem}

\begin{proof}
    Since $\Sigma^{-1} \GG_{>0}$ is a free resolution of $J$, the long exact sequence in homology associated to the mapping cone reduces to
    \[
    \xymatrix{
    0 \ar[r] & H_1\!\left(\Cone{\mu^z}\right) \ar[r] & J \ar[r]^-{\cdot z} & J \ar[r] & H_0\!\left(\Cone{\mu^z}\right) \ar[r] & 0.
    }
    \]
    We simultaneously get $H_1\!\left(\Cone{\mu^z}\right) = 0$ and $H_0\!\left(\Cone{\mu^z}\right) \cong J/zJ$ by noting $z$ is a non-zero divisor and thus multiplication by $z$ has trivial kernel. It follows that $\Cone{\mu^z} \simeq J/zJ$.

    To see that the resolution $\Cone{\mu^z}$ is minimal, we note $\partial^{\Sigma^{-1}\GG_{>0}} \otimes_Q \Bbbk = 0 = \mu^{z} \otimes_Q \Bbbk$. It follows from the definition of $\partial^{\Cone{\mu^z}}$ that $\partial^{\Cone{\mu^z}} \otimes \Bbbk = 0$, and thus $\Cone{\mu^z}$ is minimal.
\end{proof}

As a result of Lemma \ref{lem:gPrime}, we get the following description of $\GG'$:
\[\xymatrix{
\GG' := \cdots \ar[r] & {\begin{matrix} \GG_4 \\ \oplus \\ \GG_3 \end{matrix}} \ar[rr]^-{\left( \begin{smallmatrix} -\partial_4^{\GG} & \mu^z \\ 0 & \partial_3^{\GG} \end{smallmatrix} \right)} && {\begin{matrix} \GG_3 \\ \oplus \\ \GG_2 \end{matrix}} \ar[rr]^-{\left( \begin{smallmatrix} -\partial_3^{\GG} & \mu^z \\ 0 & \partial_2^{\GG} \end{smallmatrix} \right)} && {\begin{matrix} \GG_2 \\ \oplus \\ \GG_1 \end{matrix}} \ar[rr]^-{\left( \begin{smallmatrix} -\partial_2^{\GG} & \mu^z \\ 0 & 0 \end{smallmatrix} \right)} && {\begin{matrix} \GG_1 \\ \oplus \\ 0 \end{matrix}} \ar[r] & 0
.}\]
We capitalize on the structure of $\GG'$ to build a chain map $\Psi: \GG' \to \FF$ such that $\Cone{\Psi} \simeq Q/I(\Gamma)$. Specifically, we see that (the truncation of) $\GG$ is a direct summand of $\GG'$. We build $\Psi$ by first building $\Phi: \GG \to \FF$. We return our attention to Lemma \ref{lem:zColon}, which starts the process of constructing $\Phi$.  

Since $(z) = (I:_Q J)$, we get a well-defined $Q$-module homomorphism $Q/J \to Q/I$ given by $q + J \mapsto zq + I$. This homomorphism lifts to a chain map $\Phi: \GG \to \FF$. We refer to $\Phi$ as ``$z$-ification,'' a term that gets its name from the explicit definition of $\Phi$ given in Proposition \ref{prop:PhiDef}. To understand the proposition, we utilize the subsets definition of the Taylor resolution (Definition \ref{taylordef}) along with the following definition and notation.

\begin{remark}
    To use the subsets definition of the Taylor resolution, we need total orderings on $G(I)$ and $G(J)$. For $G(I)$ we define $zx_i < zx_j$ if and only if $i < j$. Similarly, for $G(J)$, we define $x_i y_{i,\ell} < x_j y_{j,p}$ if and only if $i < j$ or $i = j$ and $\ell < p$.
\end{remark}

\begin{defn}
    Let $W \subseteq G(J)$ and let $w = x_i y_{i,j} \in W$. The \emph{$z$-ification of $w$} is the monomial $w_z := zx_i$. Moreover, the \emph{$z$-ification of $W$} is the (ordered) multi-set (a set allowing repeated entries) given by \[W_z := \{w_z: w \in W\} \subseteq G(I).\] Note that the elements of $W_z$ occur in the same order as the corresponding terms occur in $W$, that is, if $v,w \in G(J)$, then $v < w$ in $G(J)$ if and only if $v_z \leq w_z$ in $G(I)$. 
\end{defn}  

Recall from Example \ref{taylorex} the notation $\sigma(V,W)=|\{(v,w)\in V \times W : v>w\}|$.  

\begin{lem}\label{lem:zSigma}
    Let $V,W \subseteq G(J)$. If neither of the multisets $V_z$ and $W_z$ has repeated elements and $V_z \cap W_z = \emptyset$, then $\sigma(V,W) = \sigma(V_z,W_z)$.
\end{lem}

\begin{proof}
    Since $V_z$ does not have any repeated elements, each $v_z \in V_z$ can be tracked back to exactly one $v \in V$. The same applies to $W_z$. Since $V_z \cap W_z = \emptyset$, we also have $V \cap W = \emptyset$. Thus, $\sigma(V_z,W_z)$ and $\sigma(V,W)$ are both well-defined, and the desired equality follows from below:
    \begin{align*}
        \sigma(V,W) &= \left|\{ (v,w) \in V \times W : v > w  \} \right| \\
         &= \left|\{ (v_z,w_z) \in V_z \times W_z : v_z > w_z  \} \right| \\
         &= \sigma(V_z,W_z). \qedhere
    \end{align*}
\end{proof}

We expand on the notation associated with the Taylor resolution.

\begin{notn}
    Recall that if $W \subseteq G(J)$, then $m_W = \lcm\{x_iy_{i,j} \in W\}$. We further define \[x_W := \lcm\{x_i: x_iy_{i,j} \in W\} \hspace{.5cm} \hbox{ and } \hspace{.5cm} y_W := \lcm\{y_{i,j} : x_i y_{i,j} \in W\}.\] We immediately get $m_W = x_W y_W$.
\end{notn}

We extend these definitions to $V \subseteq G(I)$ by setting $m_V := \lcm\{zx_i \in V\}$ and $x_V := \lcm\{x_i : zx_i \in V \}$. In this case, we have $m_V = z x_V$. By further extending this notation to multisets, we find that for any $W \in G(J)$, we have $x_W = x_{W_z}$. This is crucial for Proposition \ref{prop:PhiDef}.

Recall that $\FF$ and $\GG$ are Taylor resolutions.  The standard basis elements of $\FF$ and $\GG$ are denoted $f_V$ and $g_W$, respectively, where $V \subseteq G(I)$ and $W \subseteq G(J)$. 

\begin{prop}\label{prop:PhiDef}
    Let $\emptyset \neq W \subseteq G(J)$.  The chain map $\Phi: \GG \to \FF$ is given by $\Phi_0(1) = z$ and $\Phi(g_W) = y_W f_{W_z}$, where $f_{W_z} = 0$ if there exists some $i$ such that $zx_i$ occurs more than once in $W_z$.  
\end{prop}

\begin{proof}
    We first consider elements of $\GG_1$, which has basis $\{g_{\{x_iy_{i,j}\}} : x_iy_{i,j} \in J\}$. In this case, we find
    \begin{align*}
        \partial^{\FF}_1\!\left( \Phi_1\!\left(g_{\{x_iy_{i,j}\}}\right) \right) &= y_{i,j} \partial^{\FF}_1\!\left(f_{\{zx_i\}}\right) \\
        &= z x_i y_{i,j} \\
        &= \Phi_0\!\left(x_i y_{i,j}\right) \\
        &= \Phi_0\!\left( \partial^{\GG}_1\! \left( g_{\{ x_iy_{i,j} \}} \right) \right).
    \end{align*}
    That is, $\partial^{\FF}_1 \circ \Phi_1 = \Phi_0 \circ \partial^{\GG}_1$. We now consider higher degrees.

    First, consider $W \subseteq G(J)$ such that the multiset $W_z$ has no repeated elements. In this case, we have $f_{W_z} \neq 0$. If any subset of $W_z$ contains repeated elements, then we find $W_z$ also contains repeated elements, contradicting our initial assumption. Thus, for any $w_z \in W_z$ we also have $f_{W_z \backslash w_z} \neq 0$. We now compute the following, recalling that $x_W = x_{W_z}$:
    \begin{align*}
        \partial^{\FF}_{|W_z|}\! \left( \Phi_{|W|}\!\left( g_W \right) \right) &= y_W \partial_{|W_z|}^{\FF}\!\left( f_{W_z} \right) \\
         &= y_W \sum_{w_z \in W_z} (-1)^{\sigma(w_z,W_z)} \frac{m_{W_z}}{m_{W_z \backslash w_z}} f_{W_z \backslash w_z} \\
         &= y_W \sum_{w_z \in W_z} (-1)^{\sigma(w_z,W_z)} \frac{z x_{W_z}}{z x_{W_z \backslash w_z}} f_{W_z \backslash w_z} \\
         &= \sum_{w \in W} (-1)^{\sigma(w,W)} \frac{x_{W} y_W}{x_{W \backslash w} y_{W \backslash w}} \left(y_{W \backslash w} f_{W_z \backslash w_z}\right) \\
         &= \sum_{w \in W} (-1)^{\sigma(w,W)} \frac{m_{W}}{m_{W \backslash w} } \Phi_{|W| - 1}\!\left(g_{W \backslash w}\right) \\
         &= \Phi_{|W| - 1}\!\left( \partial^{\GG}_{|W|} \left( g_W \right) \right).
    \end{align*}
    We see that we have commutativity when $W_z$ does not have repeated elements and turn our attention to when there are repeated elements.
    
    In the case that $W_z$ has repeated elements, we defined $f_{W_z} = 0$ and find $\Phi(g_W) = 0$. This then yields $\partial^{\FF}_{|W_z|}\!\left(\Phi_{|W|}\!(g_W)\right) = 0$. Therefore, to finish the proof that $\Phi$ is a chain map, we prove that $\Phi_{|W| - 1}\!\left(\partial^{\GG}_{|W|}\!(g_W)\right) = 0$ whenever $W_z$ has repeats. To do this, we note that $W_z$ has repeated element $zx_i$ if and only if there exist distinct $j_1$ and $j_2$ such that $x_i y_{i,j_1}, x_i y_{i,j_2} \in W$. For simplicity, and without loss of generality, we re-index so that we are considering the case of $x_1 y_{1,1}, x_1 y_{1,2} \in W$. We find \[\partial^{\GG}_{|W|}(g_{W}) = y_{1,1} g_{W \backslash x_1 y_{1,1}} - y_{1,2} g_{W \backslash x_1 y_{1,2}} + \sum_{w \in W \backslash \{x_1y_{1,1}, x_1y_{1,2}\}} (-1)^{\sigma(w,W)} \frac{m_W}{m_{W \backslash w}} g_{W \backslash w}. \] Since $x_1y_{1,1}, x_1y_{1,2} \in W \backslash w$ for all the terms in the final summation notation, applying $\Phi$ to the summation yields 0 and thus
    \begin{align*}
        \Phi_{|W| - 1}\!\left( \partial^{\GG}_{|W|} (g_W) \right) &= y_{1,1} \Phi_{|W| - 1}\left(g_{W \backslash x_1y_{1,1}}\right) - y_{1,2} \Phi_{|W| - 1}\left(g_{W \backslash x_1y_{1,2}}\right) \\
         &= y_{1,1} y_{W \backslash x_1y_{1,1}} f_{(W\backslash x_1y_{1,1})_z} - y_{1,2} y_{W \backslash x_1y_{1,2}} f_{(W\backslash x_1y_{1,2})_z} \\
         &= y_{1,1} \frac{y_W}{y_{1,1}} f_{W_z \backslash x_1 z} - y_{1,2} \frac{y_W}{y_{1,2}} f_{W_z \backslash x_1 z} \\
         &= 0.
    \end{align*}
    Thus, $\partial_i^{\FF} \circ \Phi_i = \Phi_{i - 1} \circ \partial_i^{\GG}$, and so $\Phi$ is a chain map lifting multiplication by $z$.
\end{proof}

We now construct the chain map $\Psi : \GG' \to \FF$ whose mapping cone is at the center of our focus. Note that elements of $\GG'$ are ordered pairs $(g_1,g_2)$, where $g_1, g_2 \in \GG$ with $|g_1|, |g_2| > 0$.

\begin{thm}\label{thm:diam4res}
    Let $\GG' = \Cone{\mu^z}$ as in Lemma \ref{lem:gPrime}, and let $\Phi$ be the chain map from Proposition~\ref{prop:PhiDef}. If $\Psi: \GG' \to \FF$ is defined by \[\Psi_0 = \begin{pmatrix} \partial^{\GG}_1 & 0 \end{pmatrix} \hbox{ and } \Psi_i = \begin{pmatrix} 0 & \Phi_i \end{pmatrix} \hbox{ for } i > 0, \]
    then $\Psi$ is a chain map and $\Cone{\Psi}$ is the minimal free resolution of $Q/ I(\Gamma)$.
\end{thm}

\begin{proof}
    We first show that $\Psi$ is a chain map.  After doing so, we use Approach \ref{rem:SES} to argue that $\Cone{\Psi}$ is a free resolution.

    We wish to show that $\partial^{\FF}_{i+1} \circ \Psi_{i+1} = \Psi_{i} \circ \partial^{\GG'}_{i}$ for all $i \geq 0$. We start with $i = 0$:
    \begin{align*}
        \Psi_0 \circ \partial^{\GG'}_1 &= \begin{pmatrix} \partial^{\GG}_1 & 0 \end{pmatrix} \circ \begin{pmatrix} -\partial^{\GG}_2 & \mu^z \\ 0 & \partial^{\GG}_1 \end{pmatrix} \\
         &= \begin{pmatrix} - \partial^{\GG}_1 \circ \partial^{\GG}_2 & \partial^{\GG}_1 \circ \mu^z \end{pmatrix} \\
         &= \begin{pmatrix} 0 & z \partial^{\GG}_1 \end{pmatrix} \\
         &= \begin{pmatrix} 0 & \Phi_0 \circ \partial^{\GG}_1 \end{pmatrix} \\
         &= \begin{pmatrix} 0 & \partial_1^{\FF} \circ \Phi_1 \end{pmatrix} \\
         &= \partial^{\FF}_1 \circ \begin{pmatrix} 0 & \Phi_1 \end{pmatrix} \\
         &= \partial^{\FF}_1 \circ \Psi_1.
    \end{align*}
    Since $\Psi$ has the same form for all $i > 0$, we address the remaining cases simultaneously:
    \begin{align*}
        \Psi_i \circ \partial^{\GG'}_{i+1} &= \begin{pmatrix} 0 & \Phi_i \end{pmatrix} \circ \begin{pmatrix} -\partial^{\GG}_{i+2} & \mu^z \\ 0 & \partial^{\GG}_{i+1}\end{pmatrix} \\
         &= \begin{pmatrix} 0 & \Phi_i \circ \partial^{\GG}_{i+1} \end{pmatrix} \\
         &= \begin{pmatrix} 0 & \partial^{\FF}_{i+1} \circ \Phi_{i+1} \end{pmatrix} \\
         &= \partial^{\FF}_{i+1} \circ \begin{pmatrix} 0 & \Phi_{i+1} \end{pmatrix} \\
         &= \partial^{\FF}_{i+1} \circ \Psi_{i+1}.
    \end{align*}
    This shows $\Psi$ is a chain map, which tells us that $\Cone{\Psi}$ is a complex. Since $\GG'$ and $\FF$ are resolutions, we see that the associated long exact sequence in homology for $\Cone{\Psi}$ reduces to the exact sequence
    \[
    \xymatrix{
    0 \ar[r] & H_1\left(\Cone{\Psi}\right) \ar[r] & \underset{\cong I(\Gamma) / I}{\underbrace{\frac{J}{zJ}}} \ar[r] & \dfrac{Q}{I} \ar[r] & H_0\left(\Cone{\Psi}\right) \ar[r] & 0.
    }
    \]
    We observe that 
    \begin{align*}
        H_0\left(\Cone{\Psi}\right) &= \coker \partial^{\Cone{\Psi}}_1 \\
         &= \frac{Q}{\Img \partial^{\FF}_1 + \Img \partial^{\GG}_1} \\
         &= \frac{Q}{I + J} \\
         &= \frac{Q}{I(\Gamma)}.
    \end{align*}
    Moreover, since $Q/I(\Gamma) \cong (Q/I) / (I(\Gamma) / I)$, we have $H_1\left(\Cone{\Psi}\right) = 0$, and therefore $\Cone{\Psi} \simeq Q/I(\Gamma)$. We see that this recovers the short exact sequence from Approach \ref{rem:SES}.

    Lastly, since $\GG'$ and $\FF$ are minimal and $\Psi_i \otimes \Bbbk = 0$ for all $i$, we must have $\partial^{\Cone{\Psi}} \otimes \Bbbk = 0$. That is, $\Cone{\Psi}$ is the minimal free resolution.
\end{proof}

We now show how to express $\Cone{\Psi}$ in terms of $\FF$, $\GG$, $\mu^z$, and $\Phi$. 
 This helps establish the notation we will use when proving that $\Cone{\Psi}$ is a dg algebra.  

\begin{remark}\label{rem:structure}
    The free resolution $\Cone{\Psi}$ takes the following form:
    \[
    \xymatrix{
    \Cone{\Psi} := \cdots \ar[r] & {\begin{matrix} \FF_3 \\ \oplus \\ \GG_3 \\ \oplus \\ \GG_2 \end{matrix}} \ar[rr]^-{\left(\begin{smallmatrix} \partial_3^{\FF} & 0 & \Phi_2 \\ 0 & \partial_3^{\GG} & -\mu^z \\ 0 & 0 & -\partial_2^{\GG} \end{smallmatrix}\right)} && {\begin{matrix} \FF_2 \\ \oplus \\ \GG_2 \\ \oplus \\ \GG_1 \end{matrix}} \ar[rr]^-{\left(\begin{smallmatrix} \partial_2^{\FF} & 0 & \Phi_1 \\ 0 & \partial_2^{\GG} & -\mu^z \\ 0 & 0 & 0 \end{smallmatrix}\right)} && {\begin{matrix} \FF_1 \\ \oplus \\ \GG_1 \\ \oplus \\ 0 \end{matrix}} \ar[rr]^-{\left(\begin{smallmatrix} \partial_1^{\FF} & \partial_1^{\GG} & 0 \\ 0 & 0 & 0 \\ 0 & 0 & 0 \end{smallmatrix}\right)} && {\begin{matrix} \FF_0 \\ \oplus \\ 0 \\ \oplus \\ 0 \end{matrix}} \ar[r] & 0.
    }
    \]
\end{remark}

\begin{corollary}
     Let $\Gamma$ be a tree of diameter four with $\ell=a_1+\ldots+a_n$, so that $I(\Gamma)$ has the form in Equation \eqref{Iform}.  The $i$th Betti number of $Q/I(\Gamma)$ is
    \[
    \beta^Q_i=\begin{cases}
        1,&i=0,\\\\
        \displaystyle \ell+n&i=1,\\\\
        \displaystyle {\binom{\ell+1}{i}}+{\binom{n}{i}},&i>1.
        \end{cases}
    \]
    It follows that the projective dimension of $Q/I(\Gamma)$ is the maximum of $\ell+1$ and $n$. 
\end{corollary}

\begin{proof}
    We only remark that the $\displaystyle{\binom{\ell+1}{i}}$ comes from Pascal's Identity.  The rest is clear.  
\end{proof}

Similarly, one can deduce the graded Betti numbers from Remark \ref{rem:structure}.

\begin{corollary}
    Let $\Gamma$ be a tree of diameter four with $\ell=a_1+\ldots+a_n$, so that $I(\Gamma)$ has the form in Equation \eqref{Iform}. The graded Betti numbers of $Q/I(\Gamma)$, $Q/I$, and $Q/J$ satisfy
    \[\beta_{i,j}^Q\left(\frac{Q}{I(\Gamma)}\right) = \begin{cases} \beta_{0,j}^Q\left(\frac{Q}{I}\right) & i = 0, \\\\
    \beta_{1,j}^Q\left(\frac{Q}{I}\right) + \beta_{1,j}^Q\left(\frac{Q}{J}\right) & i = 1,\\\\
    \beta_{i,j}^Q\left(\frac{Q}{I}\right) + \beta_{i,j}^Q\left(\frac{Q}{J}\right) + \beta_{i-1,j-1}^Q\left(\frac{Q}{J}\right) & i \geq 2. \end{cases}\]
    Equivalently, one has the Betti splitting $I(\Gamma)=I+J$, i.e.,
    \[
    \beta_{i,j}^Q(I(\Gamma)) = \beta_{i,j}^Q(I) + \beta_{i,j}^Q(J) + \beta_{i-1,j}^Q(I \cap J).
    \]
\end{corollary}

\begin{proof}
The formula for the graded Betti numbers of $Q/I(\Gamma)$ follows from Remark \ref{rem:structure} with the extra twist on the last term coming from $\mu^z$ being multiplication by $z$. We now deduce the Betti splitting by considering the following two cases.
For $i=0$,
    \begin{align*}
        \beta_{0,j}^Q(I(\Gamma)) &= \beta_{1,j}^Q\left(\frac{Q}{I(\Gamma)}\right) \\
         &= \beta_{1,j}^Q\left(\frac{Q}{I}\right) + \beta_{1,j}^Q\left(\frac{Q}{J}\right) \\
         &= \beta_{0,j}^Q\left(I\right) + \beta_{0,j}^Q\left(J\right) \\
         &= \beta_{0,j}^Q\left(I\right) + \beta_{0,j}^Q\left(J\right) + \underset{=0}{\underbrace{\beta_{-1,j}^Q\left(I \cap J\right)}}.
    \end{align*}
For $i > 0$,
    \begin{align*}
        \beta_{i,j}^Q(I(\Gamma)) &= \beta_{i+1,j}^Q\left(\frac{Q}{I(\Gamma)}\right) \\
         &= \beta_{i+1,j}^Q\left(\frac{Q}{I}\right) + \beta_{i+1,j}^Q\left(\frac{Q}{J}\right) + \beta_{i,j-1}^Q\left(\frac{Q}{J}\right) \\
         &= \beta_{i,j}^Q\left(I\right) + \beta_{i,j}^Q\left(J\right) + \beta_{i-1,j-1}^Q\left(J\right) \\
         &= \beta_{i,j}^Q\left(I\right) + \beta_{i,j}^Q\left(J\right) + \beta_{i-1,j}^Q\left(zJ\right) \\
         &= \beta_{i,j}^Q\left(I\right) + \beta_{i,j}^Q\left(J\right) + \beta_{i-1,j}^Q\left(I\cap J\right). \qedhere
    \end{align*}
\end{proof}

\begin{remark}\label{rem:encapsulate}
    While we previously addressed the case of trees with diameter less than four, we note that this case is encapsulated by Theorem \ref{thm:diam4res}. Diameter-three trees are those for which, up to reindexing, $y_{i,j}$ for all $i>1$ and $y_{1,1} \neq 0$ (see Equation \ref{Iform}).

    For a diameter two tree, we have $y_{i,j} = 0$ for all $i$  and $j$. This results in $\GG = Q$ and $\GG' = 0$, and thus $\Cone{\Psi} = \FF$. This also holds for a diameter one tree, though this case is rather unexciting since the edge ideal is simply $(x_1z)$.
\end{remark}

We now turn our attention to the dg structure of $\Cone{\Psi}$. Since $\FF$ and $\GG$ are each Taylor resolutions, they are both dg algebras. To capitalize on their dg structures, we first define the following family of linear maps.

\begin{defn}
    Let $f \in \FF_j$ such that either $j > 1$ or $\partial^{\FF}\!(f) = x_i z$ for some $i$. We define $\omega_{f} : \GG \to \GG$~by \[\omega_{f}(g_1) := \begin{cases} -x_{i}g_{1}, & \partial^{\FF}(f)=x_iz,\\ 0, & |f|>1, \end{cases}\] where $|f|$ is the homological degree of $f$. 
    
    We extend $\omega_f$ to a linear map for any $f \in \FF_{>0}$ by defining \[\omega_{u f_1 + v f_2}(g) := u \omega_{f_1}(g) + v \omega_{f_2}(g)\] for all $g \in \GG$. We define $\omega_1(g) = 0$ for all $g \in \GG$.
\end{defn}

We use the map $\omega_f$ in the definition of a product structure on $\Cone{\Psi}$.

\begin{defn}\label{defn:treeProd}
    We define the following ten products on $\Cone{\Psi}$:
    \setcounter{equation}{-1}
    \begin{align}
        (q,0,0)\cdot (f_1,g_1,g_2) &:= q (f_1,g_1,g_2) = (qf_1, qg_1, qg_2) \text{ for all } q \in Q, \label{prod:ident} \\
        (f_1,0,0) \cdot (f_2,0,0) &:= (f_1f_2,0,0), \label{prod:ff} \\
        (f_1,0,0) \cdot (0,g_2,0) &:= \left( \frac{1}{z}f_1\Phi(g_2), 0, \omega_{f_1}(g_2)\right), \label{prod:fg1} \\
        (f_1,0,0) \cdot (0,0,g_2) &:= (0,0,0), \label{prod:fg2} \\
        (0,g_1,0) \cdot (0,g_2,0) &:= (0,g_1g_2,0), \label{prod:gg} \\
        (0,g_1,0) \cdot (0,0,g_2) &:= (-1)^{|g_1|}(0,0,g_1g_2), \label{prod:g1g2} \\
        (0,0,g_1) \cdot (0,0,g_2) &:= (0,0,0), \label{prod:shiftg} \\
        (0,g_1,0) \cdot (f_2,0,0) &:= \left( \frac{1}{z}\Phi(g_1)f_2, 0, (-1)^{|g_1|}\omega_{f_2}(g_1)\right), \label{prod:g1f} \\
        (0,0,g_1) \cdot (f_2,0,0) &:= (0,0,0), \label{prod:g2f} \\
        (0,0,g_1) \cdot (0,g_2,0) &:= (0,0, g_1g_2). \label{prod:g2g1} 
    \end{align}
\end{defn}

\begin{remark}
    If $(f,g_1,g_2), (f',g_1',g_2') \in \Cone{\Psi}_{>0}$, then Definition \ref{defn:treeProd} and the linearity of the direct sum combine to show the general product $(f,g_1,g_2)(f',g_1',g_2')$ can be written as
    \[\left(ff' + \frac{1}{z}f\Phi(g_1')+\frac{1}{z}\Phi(g_1)f', g_1g_1', \omega_f(g_1') + (-1)^{|g_1|}g_1g_2' + (-1)^{|g_1|}\omega_{f'}(g_1) + g_2g_1' \right).\] While a single expression initially seems preferable, it becomes cumbersome in our later proofs. As such, our proofs utilize a case-by-case approach to the product structure.
\end{remark}

The next few results and definitions provide the tools needed to show that the product in Definition \ref{defn:treeProd} yields a dg algebra structure on $\Cone{\Psi}$.  The following lemma is meant to assuage any apprehension about the presence of $1/z$ in our product definitions.

\begin{lem}
    If $f \in \FF_i$ and $g \in \GG_j$, then $\frac{1}{z} f \Phi(g) \in \FF_{i+j}$.
\end{lem}

\begin{proof}
    Since $g \in \GG_j$ and $\Phi : \GG \to \FF$ is a chain map, we have $\Phi(g) \in \FF_j$. It then follows that $f\Phi(g) \in \FF_{i+j}$. Thus, we need only address the presence of $1/z$.

    Without loss of generality, suppose $f = f_V$ and $g = g_W$ for some $V \subseteq G(I)$ and $W \subseteq G(J)$. By Proposition \ref{prop:PhiDef}, we have $\Phi(g_W) = y_W f_{W_z}$, where $f_{W_z} = 0$ if $W_z$ has repeated elements. In that case, we have \[\frac{1}{z} f_V \Phi(g_W) = \frac{y_W}{z} f_V \cdot 0 = 0.\] On the other hand, if $W_z$ has no repeated elements, then \[\frac{1}{z} f_V \Phi(g_W) = \frac{y_W}{z} f_V f_{W_z} = \begin{cases} (-1)^{\sigma(V,W_z)} y_W  f_{V \cup W_z}, & V \cap W_z = \emptyset, \\ 0, & V \cap W_z \neq \emptyset. \end{cases}\] That is, we either have $\frac{1}{z} f_V \Phi(g_W) = y_W f_{V \cup W_z}$ or $\frac{1}{z} f_V \Phi(g_W) = 0$. Since the product is well-defined in all cases, we have $ \frac{1}{z} f_V \Phi(g_W) \in \FF_{i+j}$. The result extends to any $f$ and $g$ by the linearity of $\Phi$.
\end{proof}





To prove the products in Definition \ref{defn:treeProd} satisfy the remaining properties of a dg algebra product, we introduce two lemmas. The first deepens our understanding of $\Phi$ and shows that while it is not a dg morphism, it has a predictable interaction with the multiplicative structures of $\FF$ and $\GG$.

\begin{lem}\label{lem:notDGmorph}
    The chain map $\Phi$ is not a dg morphism, but for all $g_1,g_2 \in \GG$, one has \[z \Phi(g_1g_2) = \Phi(g_1)\Phi(g_2)\] or, equivalently, $\Phi(g_1 g_2) = \frac{1}{z} \Phi(g_1)\Phi(g_2)$.
\end{lem}

\begin{proof}

Since $\Phi$ is $Q$-linear, it suffices to show the result for basis elements of $\GG$. For that reason, we consider nonempty $V,W \subseteq G(J)$ and prove that $z\Phi(g_V g_W) = \Phi(g_V) \Phi(g_W)$. We start by assuming $\Phi(g_V) \Phi(g_W) = 0$ and consider three cases.

For the first case, suppose $\Phi(g_V) = 0$. Recall from Proposition \ref{prop:PhiDef} that this means that there exists some index $i$ such that $zx_i$ occurs more than once in the multiset $V_z$. That is, there exist two distinct indices $j$ and $\ell$ such that $x_iy_{i,j}, x_i y_{i,\ell} \in V \subseteq V \cup W$. This tells us that $zx_i$ occurs at least twice in the multiset $(V \cup W)_z$ and thus \[\Phi(g_V g_W) \in \Span_Q (f_{(V \cup W)_z}) = \{ 0 \}.\] That is, if $\Phi(g_V) = 0$ then $\Phi(g_V g_W) = 0$.

For the second case, suppose that $\Phi(g_W) = 0$. Using the same argument as the first case, we find that $\Phi(g_V g_W) = 0$.

For the third and final case of vanishing, suppose that $\Phi(g_V) \neq 0 \neq \Phi(g_W)$ but \[\Phi(g_V) \Phi(g_W) = \left(y_V f_{V_z}\right) \left(y_W f_{W_z}\right) = y_V y_W f_{V_z} f_{W_z} = 0.\] Since $\Phi(g_V) \neq 0 \neq \Phi(g_W)$, the multisets $V_z$ and $W_z$ do not contain repeated elements and so can be viewed as sets instead of multisets. Moreover, since $\FF$ is a Taylor resolution, the vanishing of this product means that $V_z \cap W_z \neq \emptyset$. That is, there exists some $i$ such that $zx_i \in V_z \cap W_z$. This tells us that there must exist indices $j$ and $\ell$ such that $x_i y_{i,j} \in V$ and $x_i y_{i,\ell} \in W$. If $j = \ell$, then $x_i y_{i,j} \in V \cap W$ and thus $g_V g_W = 0 \in \ker \Phi$. If $j \neq \ell$, then $zx_i$ occurs twice in the multiset $(V \cup W)_z$ and thus $g_V g_W \in \ker \Phi$.

We see that whenever $\Phi(g_V)\Phi(g_W) = 0$, we also have $\Phi(g_V g_W) = 0$. To complete the proof, we now assume that $\Phi(g_V)\Phi(g_W) \neq 0$. From our previous three cases, this means that neither $V_z$ nor $W_z$ contain repeated elements and so can be viewed as sets, as opposed to multisets. Furthermore, $\Phi(g_V)\Phi(g_W) \neq 0$ means that $V_z \cap W_z = \emptyset$, and so Lemma \ref{lem:zSigma} tells us that $\sigma(V,W) = \sigma(V_z,W_z)$.

We note that $V_z \cap W_z$ gives the repeated elements in $(V \cup W)_z$. Since $V_z \cap W_z = \emptyset$, there are no repeated elements in $(V \cup W)_z$ and thus $(V \cup W)_z = V_z \cup W_z$. We are now able to observe the following:
\begin{align*}
    \Phi(g_V) \Phi(g_W) &= y_V y_W f_{V_z} f_{W_z} \\
     &= (-1)^{\sigma(V_z,W_z)} y_V y_W \frac{m_{V_z} m_{W_z}}{m_{V_z \cup W_z}} f_{V_z \cup W_z} \\
     &= (-1)^{\sigma(V,W)} y_Vy_W \frac{zx_V zx_W}{z x_{V \cup W}} f_{(V \cup W)_z} \\
     &= (-1)^{\sigma(V,W)} z\frac{x_Vy_Vx_Wy_W}{x_{V \cup W} y_{V \cup W}} y_{V \cup W} f_{(V \cup W)_z} \\
     &= (-1)^{\sigma(V,W)} z\frac{m_Vm_W}{m_{V \cup W}} y_{V \cup W} f_{(V \cup W)_z} \\
     &= (-1)^{\sigma(V,W)} z\frac{m_Vm_W}{m_{V \cup W}} \Phi(g_{V \cup W}) \\
     &= z\Phi(g_Vg_W).
\end{align*}
Rearranging, one obtains the equivalent form $\Phi(g_V g_W) = \frac{1}{z} \Phi(g_V)\Phi(g_W)$.
\end{proof}





Our next lemma is necessary to understand how $\partial^{\Cone{\Psi}}$ interacts with the product structure.

\begin{lem}\label{lem:partialMult}
    Suppose $f \in \FF$ and $g \in \GG$. If $|f| > 0$, then \[\left(\partial^{\FF}(f),0,0\right) \cdot \left(0,g,0\right) = \begin{cases} \left(\frac{1}{z} \partial^{\FF}(f) \Phi(g),0,0 \right) & |f| > 1, \\ (0,\partial^{\FF}(f)g, 0) & |f| = 1 \end{cases}.\]
\end{lem}

\begin{proof}
    Since $\partial^{\FF}$ is $Q$-linear, it suffices to consider $f_V \in \FF$ where $V \subseteq G(I)$ with $V \neq \emptyset$. In the case of $|V| = 1$, the result follows from multiplication by scalars.
    
    Next, suppose $|V| = 2$. Without loss of generality, we may assume $V = \{x_1z, x_2z\}$. In this case, we observe the following:
    \begin{align*}
        \left(\partial^{\FF}(f_V),0,0\right) \cdot \left(0,g,0\right) &= x_2\left(f_{\{x_1z\}},0,0\right) \cdot \left(0,g,0\right) - x_1 \left(f_{\{x_2z\}},0,0\right) \cdot \left(0,g,0\right) \\
        &= x_2 \left(\frac{1}{z} f_{\{x_1z\}} \Phi(g), 0, -x_1g\right) - x_1 \left(\frac{1}{z} f_{\{x_2z\}} \Phi(g),0,-x_2g\right) \\
        &= \left(\frac{1}{z}\left(x_2 f_{\{x_1z\}} - x_1 f_{\{x_2z\}}\right)\Phi(g), 0, -x_1x_2 g + x_1x_2 g\right) \\
        &= \left(\frac{1}{z} \partial^{\FF}(f_V) \Phi(g),0,0 \right).
    \end{align*}
    Now suppose $|V| > 2$, and compute the following:
    \begin{align*}
        \left(\partial^{\FF}(f_V),0,0\right) \cdot \left(0,g,0\right) &= \sum_{v \in V} (-1)^{\sigma(v,V)} \frac{m_V}{m_{V \backslash v}}\left(f_{V\backslash v},0,0\right) \cdot \left(0,g,0\right) \\
         &= \sum_{v \in V} (-1)^{\sigma(v,V)} \frac{m_V}{m_{V \backslash v}} \left(\frac{1}{z} f_{V \backslash v} \Phi(g), 0, 0\right) \\
         &= \left(\frac{1}{z} \left(\sum_{v \in V} (-1)^{\sigma(v,V)} \frac{m_V}{m_{V \backslash v}} f_{V \backslash v} \right) \Phi(g), 0, 0\right) \\
         &= \left(\frac{1}{z} \partial^{\FF}(f_V) \Phi(g),0,0 \right).
    \end{align*}
    Since $|V| \geq 1$ if and only if $V \neq \emptyset$, the proof is complete.
\end{proof}

Using Lemmas \ref{lem:notDGmorph} and \ref{lem:partialMult}, we can now prove the following theorem.

\begin{thm}\label{thm:ConeProd}
    The product in Definition \ref{defn:treeProd} is graded commutative, associative, and satisfies the Leibniz rule. Consequently, the product defines a differential graded algebra structure on $\Cone{\Psi}$.
\end{thm}

\begin{proof}
    
    Due to the computational nature of the proof that the product is graded commutative, we relegate it to Appendix \ref{append:GC}.

    Associativity can be proved using $4^3 = 64$ cases. Of the 64 cases, 37 include at least one term of the form $(q,0,0)$ where $q \in \FF_0 = Q$. Those 37 cases automatically hold since multiplication by $(q,0,0)$ is the same as multiplication by $q$. Of the remaining cases, 17 of them vanish and we show the final 10 in Appendix \ref{append:assoc}.

    The Leibniz rule can be shown by investigating $4^2 = 16$ cases. Of the 16 cases, 7 include at least one term of the form $(q,0,0)$ where $q \in \FF_0 = Q$. Those 7 cases automatically hold since multiplication by $(q,0,0)$ is the same as multiplication by $q$. Due to the computational nature of the proof, we relegate the remaining 9 cases to Appendix \ref{append:LR}.
    
\end{proof}

We now apply Theorem \ref{thm:ConeProd} to the case of edge ideals of trees.

\begin{corollary}\label{cor:treeProd}
    The edge ideal of any tree of diameter four is minimally resolved by a differential graded algebra.
\end{corollary}

\begin{proof}
    Theorem \ref{thm:diam4res} tells us that the edge ideal of any tree of diameter four is minimally resolved by $\Cone{\Psi}$ while Theorem \ref{thm:ConeProd} tells us the resolution has the structure of a dg algebra.
\end{proof}




\section{The Pruning Process}\label{pruningsection}

In this section, we describe a process of ``pruning" the minimal free resolution $\mathbb{F}$ of $Q/I$, for any monomial ideal $I$, to the minimal free resolution of a smaller monomial ideal over a smaller polynomial ring (Definition \ref{pruningprocess}).  We show that when $I$ is squarefree, a dg algebra structure on $\mathbb{F}$ descends to a dg algebra structure on the ``pruned" resolution (Theorem \ref{pruningisdg}, Corollary \ref{iteratepruning}).  In the language of combinatorics, dg-sensitive pruning implies that if the minimal free resolution of the quotient of a polynomial ring by the facet ideal of a simplicial complex $\Delta$ admits the structure of a dg algebra, then so does the minimal free resolution of the quotient by the facet ideal of each facet-induced subcomplex of $\Delta$ over the smaller polynomial ring (Corollary \ref{simplex}).  This allows us to conclude, for example, that induced subgraphs of dg graphs are dg (Corollary \ref{main}), which is used to complete the classifications of dg trees and cycles in Section \ref{class}. 

In the case of a squarefree monomial ideal $I$, Katth\"an provides the following structure theorem for the minimal free resolution $\mathbb{F}$ of $Q/I$ when $\mathbb{F}$ admits the structure of a dg algebra \cite[Theorem 3.6]{Kat}.

\begin{thm}[Katth\"an, 2019]\label{structurethm}
Let $I$ be a squarefree monomial ideal, and suppose that the minimal $Q$-free resolution $\mathbb{F}$ of $Q/I$ admits the structure of a differential graded algebra.  Then, there is an isomorphism of differential graded algebras $\mathbb{F} \cong \mathbb{T}/\mathbb{J}$, where $\mathbb{T}$ is the Taylor resolution of $Q/I$ and $\mathbb{J}$ is a dg ideal of $\mathbb{T}$.  
\end{thm}

Theorem \ref{structurethm} plays a key role in showing that the ``pruning process" (Definition \ref{pruningprocess}) is ``dg-sensitive."  We now introduce a dg ideal $\mathbb{I}$ of the Taylor resolution which combines with the dg ideal $\mathbb{J}$ from Theorem \ref{structurethm} to later formalize the notion of ``pruning."

\begin{lem}\label{principaldgideal}
   Let $I$ be a monomial ideal of a polynomial ring $Q$ in which $x_k$ is a variable.  Let $\mathbb{T}$ be the Taylor resolution of $Q/I$.  The subcomplex $\mathbb{I}$ of $\mathbb{T}$ given by
    \[
    \mathbb{I}:=\bigoplus 0 \to Qe_V \to Q\partial(e_V) \to 0,
    \]
    where the direct sum is taken over all $V \subseteq G(I)$ such that $V$ contains a monomial divisible by $x_k$, is a dg ideal of $\mathbb{T}$.  
\end{lem}

\begin{proof}
    For a basis element $e_U$ of $\mathbb{T}$, we assume that $e_Ue_V \neq 0$ and notice that $e_Ue_V$ is a multiple of $e_{U \cup V}$.  Since $U \cup V$ contains a monomial divisible by $x_k$, we have that $e_{U \cup V} \in \mathbb{I}$.  So, $e_Ue_V \in \mathbb{I}$.  Using the Leibniz rule,
    \[
    \partial(e_Ue_V)=\partial(e_U)e_V+(-1)^{|U|}e_U\partial(e_V).
    \]
    Since $e_Ue_V$ is in $\mathbb{I}$, so is $\partial(e_Ue_V)$.  By an argument similar to the above, $\partial(e_U)e_V \in \mathbb{I}$.  Thus, $e_U\partial(e_V)$ is in $\mathbb{I}$, and we have that $\mathbb{I}$ is a dg ideal of $\mathbb{T}$. 

\end{proof}

The image of a dg ideal of $\mathbb{T}$ in a quotient of $\mathbb{T}$ by a dg ideal is again a dg ideal.  

\begin{lem}\label{dgprojection}
    Let $\mathbb{J}$ be a dg ideal of $\mathbb{T}$.  Set $\mathbb{F}=\mathbb{T}/\mathbb{J}$, and let $\pi:\mathbb{T} \to \mathbb{F}$ be the quotient map.  Then, the subcomplex $\pi(\mathbb{I})$ of $\mathbb{F}$ is a dg ideal of $\mathbb{F}$, where $\mathbb{I}$ is the dg ideal of $\mathbb{T}$ from Lemma \ref{principaldgideal}.
\end{lem}

\begin{proof}
    Since $\mathbb{J}$ is a dg ideal of $\mathbb{T}$, $\pi$ is a dg morphism. Since $\mathbb{I}$ is a dg ideal of $\mathbb{T}$ and $\pi$ is a surjective dg morphism, we have \[\pi(\mathbb{I}) = \pi(\mathbb{I} \mathbb{T}) = \pi(\mathbb{I}) \pi(\mathbb{T}) = \pi(\mathbb{I}) \FF.\] Thus, $\pi(\mathbb{I})$ is a dg ideal of $\FF$.
\end{proof}

We now describe a process, the ``pruning process," which we use to show that simplicial complexes with non-dg facet-induced subcomplexes cannot themselves be dg.  The following definition/algorithm comes from \cite{boocher}.

\begin{defn}[Boocher, 2012]\label{pruningprocess}
    Let $\mathbb{F}$ be a complex of free $Q$-modules, where $Q$ is a polynomial ring, with choice of bases so that each differential is a matrix $A_i$.  Let $Z$ be a subset of the variables of $Q$.  The \textit{pruning of $\mathbb{F}$ with respect to $Z$} is the complex $P(\mathbb{F},Z)$ of $Q/\la Z\ra$-modules obtained from $\mathbb{F}$ by the following algorithm:
    \begin{enumerate}
        \item Let $i=1$.
        \item For $i \leq \max\{j \mid \mathbb{F}_j \neq 0\}$, do:
        \begin{enumerate}
            \item In the matrix $A_i$, set all variables in $Z$ equal to zero.  Set $A_i$ equal to this new matrix, and set $U$ equal to the set indexing which columns of $A_i$ are identically zero.
            \item Replace $\{A_{i+1},\mathbb{F}_i,A_i\}$ with new maps and modules obtained by deleting rows, basis elements, and columns, respectively, corresponding to $U$.  
        \end{enumerate}
        \item Let $i=i+1$.  
    \end{enumerate}
\end{defn}

As an example, we prune the minimal $Q$-free resolution of $Q/I(L(1,1,1))$ to the minimal $Q/(y_1)$-free resolution of $\dfrac{Q/(y_1)}{I(L(1,1,0))}$.  

\begin{ex}\label{pruningex}
    Let $Q=\Bbbk[x,y,x_1,y_1,z_1]$.  The minimal $Q$-free resolution $\mathbb{F}$ of $Q/I(L(1,1,1))$ is
    \[
    0 \xrightarrow{} Q^2 \xrightarrow{\left[\begin{smallmatrix}
        y_1&0\\0&x_1\\0&-z_1\\-z_1&0\\x&0\\0&y
    \end{smallmatrix}\right]} Q^6 \xrightarrow{\left[\begin{smallmatrix}
        -z_1&-z_1&-x_1&-y_1&0&0\\x&0&0&0&-y_1&0\\0&y&0&0&0&-x_1\\0&0&y&0&0&z_1\\0&0&0&x&z_1&0
    \end{smallmatrix}\right]} Q^5 \xrightarrow{\begin{bmatrix}
        xy&yz_1&xz_1&xx_1&yy_1
    \end{bmatrix}} Q \xrightarrow{} 0.
    \]
    We prune by setting $y_1$ equal to zero.  The first loop of the pruning process deletes the fifth column of the first differential and thus the fifth row of the second differential, yielding the following sequence:
    \[
    0 \xrightarrow{} Q^2 \xrightarrow{\begin{bmatrix}
        y_1&0\\0&x_1\\0&-z_1\\-z_1&0\\x&0\\0&y
    \end{bmatrix}} Q^6 \xrightarrow{\begin{bmatrix}
        -z_1&-z_1&-x_1&-y_1&0&0\\x&0&0&0&-y_1&0\\0&y&0&0&0&-x_1\\0&0&y&0&0&z_1
    \end{bmatrix}} Q^4 \xrightarrow{\begin{bmatrix}
        xy&yz_1&xz_1&xx_1
    \end{bmatrix}} Q \xrightarrow{} 0.
    \]
    The second loop of the pruning process deletes the fourth and fifth columns of the second differential and thus the fourth and fifth rows of the third differential:
    \[
    0 \xrightarrow{} Q^2 \xrightarrow{\begin{bmatrix}
        y_1&0\\0&x_1\\0&-z_1\\0&y
    \end{bmatrix}} Q^4 \xrightarrow{\begin{bmatrix}
        -z_1&-z_1&-x_1&0\\x&0&0&0\\0&y&0&-x_1\\0&0&y&z_1
    \end{bmatrix}} Q^4 \xrightarrow{\begin{bmatrix}
        xy&yz_1&xz_1&xx_1
    \end{bmatrix}} Q \xrightarrow{} 0.
    \]
    Finally, the pruning process deletes the first column of the third differential:
    \[
    0 \xrightarrow{} Q \xrightarrow{\begin{bmatrix}
        0\\x_1\\-z_1\\y
    \end{bmatrix}} Q^4 \xrightarrow{\begin{bmatrix}
        -z_1&-z_1&-x_1&0\\x&0&0&0\\0&y&0&-x_1\\0&0&y&z_1
    \end{bmatrix}} Q^4 \xrightarrow{\begin{bmatrix}
        xy&yz_1&xz_1&xx_1
    \end{bmatrix}} Q \xrightarrow{} 0.
    \]
    Upon replacing $Q$ with $Q/(y_1)$, the resulting sequence $P(\mathbb{F},\{y_1\})$ is the minimal $Q/(y_1)$-free resolution of $\dfrac{Q/(y_1)}{I(L(1,1,0))}$:
    \[
    0 \xrightarrow{} Q/(y_1) \xrightarrow{\begin{bmatrix}
        0\\x_1\\-z_1\\y
    \end{bmatrix}} Q/(y_1)^4 \xrightarrow{\begin{bmatrix}
        -z_1&-z_1&-x_1&0\\x&0&0&0\\0&y&0&-x_1\\0&0&y&z_1
    \end{bmatrix}} Q/(y_1)^4 \xrightarrow{\begin{bmatrix}
        xy&yz_1&xz_1&xx_1
    \end{bmatrix}} Q/(y_1) \xrightarrow{} 0.
    \]
\end{ex}

Inherent in Definition \ref{pruningprocess} is that $P(\mathbb{F},Z)$ is a complex.  If the entries of the differentials of $\mathbb{F}$ are in the maximal ideal of $Q$, then so are the entries of the differentials of $P(\mathbb{F},Z)$.  That is, minimal complexes prune to minimal complexes.  What is surprising is the following result that Example \ref{pruningex} is true more generally \cite[Theorem 2.3]{boocher}.  

\begin{thm}[Boocher, 2012]\label{pruningmons}
    Let $I$ be a monomial ideal of a polynomial ring $Q$, and let $\mathbb{F}$ be the minimal free resolution of $Q/I$ over $Q$.  If $Z$ is a subset of the variables of $Q$, then $P(\mathbb{F},Z)$ is the minimal free resolution of $Q/I \otimes Q/\la Z \ra$ as a $Q/\la Z\ra$-module.  
\end{thm}

We now show that Boocher's pruning process is dg-sensitive in the case of squarefree monomial ideals.  

\begin{thm}\label{pruningisdg}
    Fix the following notation:
    \begin{itemize}
        \item $Q=\Bbbk[x_1,\ldots,x_n]$
        \item $I$ is a squarefree monomial ideal of $Q$
        \item $\FF$ is the minimal $Q$-free resolution of $Q/I$ 
        \item $R'=\Bbbk[x_1,\ldots,x_{k-1},x_{k+1},\ldots,x_n]=Q/(x_k)$
        \item $I'=I/(x_k)$
    \end{itemize}

    
    Suppose $\mathbb{F}$ admits the structure of a dg algebra over $Q$, so that $\mathbb{F} \cong \mathbb{T}/\mathbb{J}$ for some dg ideal $\mathbb{J}$ of $\mathbb{T}$ (see Theorem \ref{structurethm}). Applying Boocher's pruning process (Definition \ref{pruningprocess}) with $Z=\{x_k\}$, we have that $P(\mathbb{F},Z)$, the minimal $R$-free resolution of $R/I'$, is isomorphic to the complex obtained from $\mathbb{F}$ by modding out by the image of the dg ideal $\mathbb{I}$ from Lemma \ref{principaldgideal} under the quotient map $\mathbb{T} \onto \mathbb{F} \cong \mathbb{T}/\mathbb{J}$ and tensoring down to $R$.  Thus, the minimal free resolution $P(\mathbb{F},Z)$ of $R/I'$ over $R$ admits the structure of a differential graded algebra.
\end{thm}

\begin{proof}

    We break down the proof into steps.  

    \textbf{Step 1: generate $\mathbb{I}$ by pruning $x_k$.} The differentials we list in what follows are those of $\mathbb{T}$.  Set $x_k$ equal to zero for pruning.  In $\partial_0$, the columns with entries divisible by $x_k$ are deleted.  The corresponding rows of $\partial_1$ are also deleted, killing the terms of the images of the $e_W$, $|W|=2$, that are supported on $e_V$ for some $V \subseteq G(I)$ containing a monomial divisible by $x_k$.  
    
    The columns of $\partial_i$ corresponding to basis elements $e_W$, where $W$ contains a monomial divisible by $x_k$, have nonzero entries corresponding to $e_{W'}$, where $W' \subseteq W$ with $|W'|=|W|-1$.  The entries divisible by $x_k$ go to zero in the quotient.  The other entries correspond to those $W'$ containing monomials divisible by $x_k$.  The columns of $\partial_{i-1}$ corresponding to these $W'$ are deleted in the previous step, and thus these rows in $\partial_i$ are also deleted, leaving a column of zeros in $\partial_i$.  

    The columns of $\partial_i$ corresponding to basis elements $e_W$, where $W$ does not contain a monomial divisible by $x_k$, has nonzero, not-divisible-by-$x_k$ entries before deleting rows.  But the rows in which these entries appear correspond to columns in $\partial_{i-1}$ corresponding to $e_{W'}$, $W' \subseteq W$, $|W'|=|W|-1$, satisfying that no monomial in $W'$ is divisible by $x_k$.  So, these columns are not deleted by the previous step, and thus their corresponding rows in $\partial_i$ are not deleted.  

    The columns and rows crossed out by pruning are the columns and rows corresponding to basis elements $e_W$ of $\mathbb{T}$ with $W$ containing a monomial divisible by $x_k$.  That is, we are setting 
    \[
    0 \to Qe_W \to Q\partial(e_W) \to 0
    \]
    equal to zero for each such $W$, and thus the sum of these short exact sequences generate the dg ideal $\mathbb{I}$ from Lemma \ref{principaldgideal}.
    
    
    
    \textbf{Step 2: tensoring with $R$ yields a new Taylor resolution}. 
   
    Upon tensoring $\mathbb{T}/\mathbb{I}$ with $R$, we get
    \[
    \mathbb{T}/\mathbb{I} \otimes_Q R \cong \dfrac{\mathbb{T} \otimes_Q R}{\mathbb{I} \otimes_Q R},
    \]
    which is the Taylor resolution on the minimal generators of $I'$.

    \textbf{Step 3: modding out by $\mathbb{J}$ yields $P(\mathbb{F},Z)$}. By definition, $P(\mathbb{F},Z) \cong \mathbb{F}/\mathbb{I}' \otimes_Q R$, where $\mathbb{I}'$ is generated by elements of $\mathbb{F}$ that have multidegree divisible by $x_k$. Since $\pi : \mathbb{T} \to \mathbb{F}$ is a surjection that preserves multidegrees, we have $\pi(\mathbb{I}) = \mathbb{I}'$. Thus, 
    \[
    \frac{\dfrac{\mathbb{T}}{\mathbb{I}} \otimes_Q R}{\dfrac{\mathbb{J} + \mathbb{I}}{\mathbb{I}} \otimes_Q R} \cong \frac{\mathbb{T}}{\mathbb{I} + \mathbb{J}} \otimes_Q R \cong \frac{\mathbb{F}}{\pi(\mathbb{I})} \otimes_Q R \cong \frac{\mathbb{F}}{\mathbb{I}'} \otimes_Q R \cong P(\mathbb{F},Z).
    \]
    
    By Lemmas \ref{principaldgideal} and \ref{dgprojection}, $P(\mathbb{F},Z)$ admits the structure of a dg algebra.  

    By Theorem \ref{pruningmons}, the pruning process applied to the minimal free resolution of a monomial ideal produces the minimal free resolution of the pruned ideal over the pruned ring, and so we are done.    
\end{proof}

Pruning iteratively yields the following corollary. 

\begin{cor}\label{iteratepruning}
    If $I$ is a squarefree monomial ideal of $Q=\Bbbk[x_1,\ldots,x_n]$ and the minimal $Q$-free resolution $\mathbb{F}$ of $Q/I$ admits the structure of a differential graded algebra, then so does the minimal $R$-free resolution of $I'$, where $R$ and $I'$ are the quotients of $Q$ and $I$, respectively, by an ideal generated by a subset of $\{x_1,\ldots,x_n\}$.  
\end{cor}

Squarefree monomial ideals are in one-to-one correspondence with facet ideals of simplicial complexes, a generalization of edge ideals introduced by Faridi in \cite{Faridi}.

\begin{defn}\label{facetideal}
    Let $\Delta$ be a simplicial complex on vertices $1,\ldots,n$.  The \textit{facet ideal} $\mathcal{F}(\Delta)$ is the ideal of $Q=\Bbbk[x_1,\ldots,x_n]$ generated by the squarefree monomials $x_{i_1}\cdots x_{i_s}$, where $\{i_1,\ldots,i_s\}$ is a facet of $\Delta$.  
\end{defn}

\begin{remark}
    In Definition \ref{facetideal}, taking $\Delta$ to be a graph recovers Definition \ref{edgeideal}, i.e., the facet ideal of a graph (with no isolated vertices) is its edge ideal.  
\end{remark}

\begin{defn}\label{facetinduced}
    Let $\Delta$ be a simplicial complex on $V=\{v_1,\ldots,v_n\}$, and let $W \subseteq V$.  The \textit{facet-induced subcomplex} $\Delta'$ of $\Delta$ on $W$ is the simplicial complex generated by the facets $F$ of $\Delta$ such that $F \cap (V\backslash W) = \emptyset$.   
\end{defn}

Note that facet-induced subcomplexes and induced subcomplexes need not coincide.

\begin{ex}
    The induced subcomplex of the 2-simplex on any two of its vertices is a 1-simplex, whereas any proper facet-induced subcomplex of the 2-simplex is the empty set.    
\end{ex}

Theorem \ref{pruningisdg} yields the following result about facet ideals of facet-induced subcomplexes.

\begin{cor}\label{simplex}
     Let $\Delta$ be a simplicial complex such that the minimal $Q$-free resolution of $Q/\mathcal{F}(\Delta)$ admits the structure of a differential graded algebra, where $\mathcal{F}(\Delta)$ is the facet ideal of $\Delta$ and $Q$ is the ambient polynomial ring of $\Delta$.  If $\Delta'$ is a facet-induced subcomplex of $\Delta$ and $Q'$ is the ambient polynomial ring of $\Delta'$, then the minimal $Q'$-free resolution of $Q'/\mathcal{F}(\Delta')$ admits the structure of a differential graded algebra.
\end{cor}

Our next result about induced subgraphs follows from Corollary \ref{simplex} by restricting $\Delta$ to be a graph.  

\begin{cor}\label{main}
    Let $G$ be a graph such that the minimal $Q$-free resolution of $Q/I(G)$ admits the structure of a differential graded algebra, where $Q$ is the ambient polynomial ring of $G$.  If $G'$ is a induced subgraph of $G$ and $Q'$ is the ambient polynomial ring of $G'$, then the minimal $Q'$-free resolution of $Q'/I(G')$ admits the structure of a differential graded algebra.  
\end{cor}

\begin{proof}
    The classes of facet-induced subgraphs of a (connected) graph $G$ and induced subgraphs of $G$ only differ in that the latter class contains graphs with isolated vertices.  The edge ideal of a vertex is trivially dg, and the minimal free resolution of the edge ideal of a disconnected graph is the tensor product of the minimal free resolutions of the edge ideals of each connected component.  It follows that induced subgraphs of dg graphs are dg.      
\end{proof}

In particular, Corollary \ref{main} yields the following obstructions to the existence of dg algebra structure.  

\begin{corollary}\label{5+}
    Let $I$ be the edge ideal of a tree of diameter at least five.  The minimal $Q$-free resolution of $Q/I$ does not admit the structure of a differential graded algebra.  
\end{corollary}

\begin{proof}
    By iteratively pruning leaves, any tree of diameter at least five may be pruned to the path $P_6$ of diameter five, which is not dg (Proposition \ref{5path}). The contrapositive of Corollary \ref{main} then gives the~result.
\end{proof}

\begin{corollary}\label{cycle+}
    Let $C_n$ be the cycle on $n \geq 7$ vertices.  The minimal $Q$-free resolution of $Q/I(C_n)$ does not admit the structure of a differential graded algebra.  
\end{corollary}

\begin{proof}
    Any such cycle prunes to a tree of diameter at least five, which is not dg  by Corollary \ref{5+}.  
\end{proof}

\section{Classifications}\label{class}

In this section, we package our previous results to classify the trees and cycles $G$ such that $Q/I(G)$ is minimally resolved by a dg algebra.  We begin with trees.  The existence results in Sections \ref{Lyubezniksection} and \ref{d4section} and the obstructions from the pruning process in Section \ref{pruningsection} come together to yield the following classification.

\begin{thm}\label{classification}
    Let $\Gamma$ be a tree of diameter $d$.  The minimal $Q$-free resolution of $Q/I(\Gamma)$ admits the structure of a differential graded algebra if and only if $d \leq 4$.  
\end{thm}

\begin{proof}
    Observation \ref{d012obs} combines with Corollary \ref{d3dg} and Corollary \ref{cor:treeProd} for the existence case, while Corollary \ref{5+} provides the nonexistence case.  
\end{proof}

We combine results from Section \ref{pruningsection} with the paper \cite{BE} to yield an analogous result for cycles.

\begin{thm}\label{classification2}
    Let $C_n$ be the cycle on $n$ vertices.  The minimal $Q$-free resolution of $Q/I(C_n)$ admits the structure of a differential graded algebra if and only if $n \leq 5$. 
\end{thm}

\begin{proof}

    For $n \leq 4$, the minimal free resolution of $Q/I(C_n)$ admits the structure of a dg algebra by \cite{BE}.  Example \ref{notiff} uses discrete Morse theory to show that $C_5$ is dg.  One could alternatively note the minimal free resolution of $Q/I(C_5)$ is a specific instance of a Buchsbaum-Eisenbud resolution, which admits the structure of a dg algebra \cite[Theorem 4.1]{BE}.  

    For $n=6$, the minimal $Q$-free resolution of $Q/I(C_6)$ does not admit the structure of a dg algebra. Indeed, by \cite[Theorem 4.1]{Kat}, if the minimal free resolution of $Q/I(C_6)$ admits the structure of a dg algebra, then its Betti vector (its Betti numbers listed as a vector in increasing homological order) must be the $f$-vector of a simplicial complex which is a cone.  A quick computation in \textit{Macaulay2} \cite{M2} shows that the Betti vector of $Q/I(C_6)$ is $(1,6,9,6,2)$.  Using the Kruskal-Katona Theorem (see, for example, \cite[Theorem 2.1]{KK}) and the facts that
   \[
   2={\binom{4}{4}}+{\binom{3}{3}} \, \text { and }\, {\binom{4}{3}}+{\binom{3}{2}}=4+3=7 > 6,
   \]
   we find that $(1,6,9,6,2)$ is not the $f$-vector of a simplicial complex.

   Finally, Corollary \ref{cycle+} takes care of the case when $n \geq 7$.
\end{proof}

\begin{remark}\label{dgstructures}
    It is worth noting that explicit dg algebra structures on the minimal free resolutions of $Q/I(C_n)$ are known for $n=3,4,5$.  For $n=3$, the structure is that of the Hilbert-Burch complex \cite{herzog}.  The $n=4$ case follows from setting $\mathcal{X}$ and $\mathcal{Y}$ from \cite[Corollary 4.6.9]{hugh} to be the Koszul complexes on $x_1, x_3$ and $x_2, x_4$, respectively.  As mentioned in the proof of \ref{classification2}, the minimal free resolution of $Q/I(C_5)$ is a Buchsbaum-Eisenbud resolution, and one can see, for example, \cite{keri} for an explicit description of the multiplication.    
\end{remark}

\appendix

\section{Graded Commutativity}\label{append:GC}


In this appendix we prove the product in Definition \ref{defn:treeProd}, which we now restate, is graded commutative.

\begin{thm*}
    Recall the products in Definition \ref{defn:treeProd}:
    \setcounter{equation}{0}
    \begin{align}
    (f_1,0,0) \cdot (f_2,0,0) &= (f_1f_2,0,0) \\
    (f_1,0,0) \cdot (0,g_2,0) &= \left( \frac{1}{z}f_1\Phi(g_2), 0, \omega_{f_1}(g_2)\right) \text{, where } \omega_{f_1}(g_2) = \begin{cases} -x_{i}g_{2}, & \text{if} \hspace{2mm} \partial^\FF(f_1) = x_iz, \\ 0, & \text{if}\hspace{2mm} |f_1|>1, \end{cases} \\
    (f_1,0,0) \cdot (0,0,g_2) &= (0,0,0) \\
    (0,g_1,0) \cdot (0,g_2,0) &= (0,g_1 g_2,0) \\
    (0,g_1,0) \cdot (0,0,g_2) &= (-1)^{|g_1|}(0,0,g_1g_2) \\
    (0,0,g_1) \cdot (0,0,g_2) &= (0,0,0) \\
    (0,g_1,0) \cdot (f_2,0,0) &= \left( \frac{1}{z}\Phi(g_1)f_2, 0, (-1)^{|g_1|}\omega_{f_2}(g_1)\right) \\
    (0,0,g_1) \cdot (f_2,0,0) &= (0,0,0) \\
    (0,0,g_1) \cdot (0,g_2,0) &= (0,0, g_1g_2).
    \end{align}
    The products are graded commutative.
\end{thm*}

\begin{proof}
    We first note that $(q,0,0)\cdot (f_1,g_1,g_2) = q (f_1,g_1,g_2)$ is automatically graded commutative since $(1,0,0)$ is the identity and $(q,0,0) = q(1,0,0)$. The products \eqref{prod:ff} and \eqref{prod:gg} inherit graded commutativity from $\FF$ and $\GG$, respectively. The vanishing of product \eqref{prod:shiftg} automatically satisfies graded commutativity.

    To understand the rest of the proof, it is helpful to recall that for $f \in \FF$ and $g \in \GG_{>0}$, we have that $|(f,0,0)| = |f|$, $|(0,g,0)| = |g|$, and $|(0,0,g)| = |g| + 1$.

    The rest of the products come in pairs, in the sense that proving one yields the other. We compare products \eqref{prod:fg1} and \eqref{prod:g1f} by considering the cases when $|f_1|=1$ and when $|f_1|\neq1$ separately. When $|f_1|=1$, we find that \looseness -1
    \begin{align*}
        (f_1,0,0)(0,g_1,0) &= \left(\frac{1}{z} f_1 \Phi(g_1), 0 , \omega_{f_1}(g_1) \right) \\
         &= (-1)^{|f_1| \cdot |g_1|}\left(\frac{1}{z} \Phi(g_1) f_1, 0 , (-1)^{|f_1| \cdot |g_1|} \omega_{f_1}(g_1) \right) \\
         &= (-1)^{|f_1| \cdot |g_1|}\left(\frac{1}{z} \Phi(g_1) f_1, 0 , (-1)^{|g_1|} \omega_{f_1}(g_1) \right) \\
         &= (-1)^{|f_1| \cdot |g_1|} (0, g_1, 0) (f_1,0,0).
    \end{align*}

Furthermore, looking at the degrees, we have $|(f_1\Phi(g_1),0,0)|= |f_1|+|g_1| = 1 +|g_1|=|(0,0,\omega_{f_1}(g_1))|$. Thus, 
 \begin{align*}
        \left|(f_1,0,0) \cdot (0,g_1,0) \right| &= \left| \left(\frac{1}{z} f_1 \Phi(g_1) , 0, \omega_{f_1}(g_1)\right) \right| \\
        &= \left| \left(\frac{1}{z}f_1\Phi(g_1),0,0\right) + (0,0,\omega_{f_1}(g_1))\right|\\
         &= \left| \left(\frac{1}{z} f_1 \Phi(g_1) , 0, 0\right) \right| \\
         &= \left| f_1 \Phi(g_1)\right| \\
         &= |f_1| + |g_1| \\
         &= |(f_1,0,0)| + |(0,g_1,0)|.
    \end{align*}
   When $|f_1| \neq 1$, we have $\omega_{f_1}(g_1) = 0$. Thus,

   \begin{align*}
        (f_1,0,0)(0,g_1,0) &= \left(\frac{1}{z} f_1 \Phi(g_1), 0 , \omega_{f_1}(g_1) \right) \\
         &= (-1)^{|f_1| \cdot |g_1|}\left(\frac{1}{z} \Phi(g_1) f_1, 0 , 0 \right) \\
         &= (-1)^{|f_1| \cdot |g_1|} (0, g_1, 0) (f_1,0,0).
    \end{align*}
  Since $\omega_{f_1}(g_1) = 0$, a similar computation of degrees holds:  
   \begin{align*}
        \left|(f_1,0,0) \cdot (0,g_1,0) \right| 
         &= \left| \left(\frac{1}{z} f_1 \Phi(g_1) , 0, 0\right) \right| \\
         &= \left| f_1 \Phi(g_1)\right| \\
         &= |f_1| + |g_1| \\
         &= |(f_1,0,0)| + |(0,g_1,0)|.
    \end{align*}

    The next pair is products \eqref{prod:fg2} and \eqref{prod:g2f}. In this case, both products vanish, and thus \[(f_1,0,0)(0,0,g_1) = (0,0,0) = (-1)^{|f_1|\cdot(|g_1| + 1)} (0,0,0) = (-1)^{|f_1|\cdot(|g_1| + 1)} (0,0,g_1)(f_1,0,0).\] The vanishing of both products also yields \[|(f_1,0,0) \cdot (0,0,g_1)| = |(f_1,0,0)| + |(0,0,g_1)|.\]

    Products \eqref{prod:g1g2} and \eqref{prod:g2g1} are the final pair. We find that \[(0,g_1,0)(0,0,g_2) = (-1)^{|g_1|} (0,0,g_1g_2) = (-1)^{|g_1| + |g_1| \cdot |g_2|} (0,0,g_2g_1) = (-1)^{|g_1| \cdot (|g_2| + 1)} (0,0,g_2)(0,g_1,0).\] For the degrees, we see that
    \begin{align*}
        |(0,g_1,0) \cdot (0,0,g_2)| &= \left|(-1)^{|g_1|} (0,0,g_1g_2)\right| \\
         &= |g_1g_2| + 1 \\
         &= |g_1| + |g_2| + 1 \\
         &= |(0,g_1,0)| + |(0,0,g_2)|.
    \end{align*}

    Lastly, consider $(f_1,g_1,g_2) \in \Cone{\Psi}_i$, where $i$ is odd. Specifically, we have $|f_1| = |g_1| = i$ is odd and $|g_2| = i - 1$ is even. Since $\FF$ and $\GG$ are each Taylor resolutions, $f_1 \in \FF_{>0}$, and $g_1,g_2 \in \GG_{>0}$, we have $f_1^2 = g_1^2 = g_2^2 = 0$.  So,

    \begin{align*}
        (f_1,g_1,g_2)^2 &= (f_1,0,0)^2 + (f_1,0,0)(0,g_1,0) + (f_1,0,0)(0,0,g_2) + (0,g_1,0)(f_1,0,0) + (0,g_1,0)^2 \\
         &\qquad  +(0,g_1,0)(0,0,g_2) + (0,0,g_2)(f_1,0,0) + (0,0,g_2)(0,g_1,0) + (0,0,g_2)^2 \\
         &= \left(f_1^2,0,0\right) + \left(\frac{1}{z}f_1\Phi(g_1),0,\omega_{f_1}(g_1)\right) + \left(\frac{1}{z}\Phi(g_1)f_1,0,(-1)^i \omega_{f_1}(g_1)\right) + (0,g_1^2,0) \\
         &\qquad + (-1)^i (0,0,g_1g_2) + (0,0,g_2g_1) \\
         &= \left(\frac{1}{z}f_1\Phi(g_1),0,\omega_{f_1}(g_1)\right) - \left(\frac{1}{z}f_1\Phi(g_1),0,\omega_{f_1}(g_1)\right) - (0,0,g_1g_2) + (0,0,g_1g_2) \\
         &= 0.
    \end{align*}
    Thus, the product is graded commutative.
\end{proof}

\section{Associativity} \label{append:assoc}
In this appendix we prove the products from Definition \ref{defn:treeProd} satisfy associativity. To do this, we recall from Lemma \ref{lem:notDGmorph} that \[\Phi(g_1g_2) = \frac{1}{z} \Phi(g_1) \Phi(g_2) \] for all $g_1, g_2 \in \GG$.

\begin{thm*}
     The products in Definition \ref{defn:treeProd} are associative.
\end{thm*}

\begin{proof}
    There are 27 total triples to check, but 17 of them equal $(0,0,0)$ which we leave to the interested reader. We thus only check the 10 non-vanishing combinations which we list below:
    \begin{enumerate}
        \item $(f_1,0,0)(f_2,0,0)(f_3,0,0)$
        \item $(0,g_1,0)(0,g_2,0)(0,g_3,0)$
        \item $(f_1,0,0)(f_2,0,0)(0,g_1,0)$
        \item $(f_1,0,0)(0,g_1,0)(f_2,0,0)$
        \item $(0,g_1,0)(f_1,0,0)(f_2,0,0)$
        \item $(f_1,0,0)(0,g_1,0)(0,g_2,0)$
        \item $(0,g_1,0)(f_1,0,0)(0,g_2,0)$
        \item $(0,g_1,0)(0,g_2,0)(f_1,0,0)$
        \item $(0,g_1,0)(0,g_2,0)(0,0,g_3)$
        \item $(0,0,g_3)(0,g_1,0)(0,g_2,0)$
    \end{enumerate}

    We note that the first two triples inherit associativity from $\FF$ and $\GG$, respectively.

    For the third nontrivial triple, recall that since $|f_1f_2|>1, \omega_{f_1f_2}(g_1)=0$, and so we find
    \begin{align*}
        [(f_1,0,0)(f_2,0,0)](0,g_1,0) &= \left(f_1f_2,0,0\right)(0,g_1,0) \\
         &= \left(\frac{1}{z}f_1f_2\Phi(g_1),0,0 \right) \\
         &= (f_1,0,0)\left(\frac{1}{z}f_2\Phi(g_1),0, \omega_{f_2}(g_1)\right) \\
         &= (f_1,0,0)[(f_2,0,0)(0,g_1,0)].
    \end{align*}

    For the fourth triple, we find
    \begin{align*}
        [(f_1,0,0)(0,g_1,0)](f_2,0,0) & = \left(\frac{1}{z}f_1\Phi(g_1),0,\omega_{f_1}(g_1)\right)(f_2,0,0) \\
        &= \left(\frac{1}{z}f_1\Phi(g_1)f_2,0,0 \right) \\
        &= (f_1,0,0)\left( \frac{1}{z}\Phi(g_1)f_2,0,\omega_{f_2}(g_1)\right) \\
        &= (f_1,0,0)[(0,g_1,0)(f_2,0,0)]
     \end{align*}

     For the fifth triple, we see that $|f_1f_2| = |f_1| + |f_2| >1$, thus $\omega_{f_1f_2}(g_1)=0$. We then find the following:
     \begin{align*}
        [(0,g_1,0)(f_1,0,0)](f_2,0,0) &= \left(\frac{1}{z}\Phi(g_1)f_1,0, (-1)^{|g_1|} \omega_{f_1}(g_1) \right)(f_2,0,0) \\
         &= \left(\frac{1}{z}\Phi(g_1)f_1f_2,0,0 \right) \\
         &= (0,g_1,0)(f_1f_2,0,0) \\
         &= (0,g_1,0)[(f_1,0,0)(f_2,0,0)]
    \end{align*}

    For the sixth triple, note if $\partial(f_1)=x_1z$, then $\omega_{f_1}(g_1g_2)= -x_1g_1g_2= \omega_{f_1}(g_1)g_2$.
    \begin{align*}
        [(f_1,0,0)(0,g_1,0)](0,g_2,0) &= \left( \frac{1}{z}f_1\Phi(g_1),0,\omega_{f_1}(g_1)\right)(0,g_2,0) \\
        &= \left(\frac{1}{z}f_1\Phi(g_1),0,0 \right)(0,g_2,0) + (0,0, \omega_{f_1}(g_1))(0,g_2,0) \\
        &= \left(\frac{1}{z^2}f_1\Phi(g_1)\Phi(g_2),0, \omega_{f_1}(g_1)g_2\right) \\
        &= \left( \frac{1}{z}f_1\Phi(g_1g_2),0,\omega_{f_1}(g_1g_2)\right)\\
        &= (f_1,0,0)(0, g_1g_2,0) \\
        &= (f_1,0,0)[(0,g_1,0)(0,g_2,0)]
    \end{align*}

    For the seventh triple, we note that $\omega_{f_1}$ is multiplication by the scalar $-x_i$, and so $\omega_{f_1}(g_1)g_2 =( -x_i)g_1g_2 = g_1(-x_ig_2)=g_1\omega_{f_1}(g_2).$ Thus,
    \begin{align*}
        [(0,g_1,0)(f_1,0,0)](0,g_2,0) &= \left(\frac{1}{z}\Phi(g_1)f_1,0, (-1)^{|g_1|}\omega_{f_1}(g_1) \right)(0,g_2,0) \\
         &= \left(\frac{1}{z^2}\Phi(g_1)f_1\Phi(g_2),0, (-1)^{|g_1|}\omega_{f_1}(g_1)g_2\right)\\
         &= \left(\frac{1}{z^2}\Phi(g_1)f_1\Phi(g_2), 0, (-1)^{|g_1|}g_1\omega_{f_1}(g_2)  \right) \\
         &= (0,g_1,0)\left( \frac{1}{z}f_1\Phi(g_2), 0, \omega_{f_1}(g_2)\right) \\
         &= (0,g_1,0)[(f_1,0,0)(0,g_2,0)].
    \end{align*}

    For the eighth triple, we find:
    \begin{align*}
        [(0,g_1,0)(0,g_2,0)](f_1,0,0) &= (0,g_1g_2,0)(f_1,0,0) = \left(\frac{1}{z}\Phi(g_1g_2)f_1, 0, (-1)^{|g_1g_2|}\omega_{f_1}(g_1g_2) \right) \\
         &= \left(\frac{1}{z^2}\Phi(g_1)\Phi(g_2)f_1, 0, (-1)^{|g_1|+|g_2|}g_1\omega_{f_1}(g_2)\right) \\
         &= (0,g_1,0)\left( \frac{1}{z}\Phi(g_2)f_1,0,(-1)^{|g_2|}\omega_{f_1}(g_2)\right) \\
         &= (0,g_1,0)[(0,g_2,0)(f_1,0,0)]
    \end{align*}

    For the ninth triple, we find:
    \begin{align*}
        [(0,g_1,0)(0,g_2,0)](0,0,g_3) &= (0,g_1g_2,0)(0,0,g_3) \\
         &= (-1)^{|g_1g_2|}(0,0,(g_1g_2)g_3) \\
         &= (-1)^{|g_1|+|g_2|}(0,0,g_1g_2g_3) \\
         &= (-1)^{|g_2|}(g_1,0,0)(0,0,g_2g_3) \\
         &= (g_1,0,0)[(0,0,g_2)(0,0,g_3)]
    \end{align*}

    For the tenth triple, we find:
    \begin{align*}
        [(0,0,g_3)(0,g_1,0)](0,g_2,0) &= (0,0,g_3g_1)(0,g_2,0) \\
         &= (0,0,g_3g_1g_2) \\
         &= (0,0,g_3)(0,g_1g_2,0) \\
         &= (0,0,g_3)[(0,g_1,0)(0,g_2,0)]
    \end{align*}

    Since all the non-vanishing triples are associative, we find the product formula is associative.
\end{proof}

\section{Leibniz Rule}\label{append:LR}

In this appendix we prove the products from Definition \ref{defn:treeProd} satisfy the Leibniz rule. To do this, we recall from Lemma \ref{lem:partialMult} that
\[\left(\partial^{\FF}(f),0,0\right) \cdot \left(0,g,0\right) = \begin{cases} \left(\frac{1}{z} \partial^{\FF}(f) \Phi(g),0,0 \right) & |f| > 1, \\ (0,\partial^{\FF}(f)g, 0) & |f| = 1. \end{cases}\]

\begin{thm*}
    The products in Definition \ref{defn:treeProd} satisfy the Leibniz rule.  
\end{thm*}

\begin{proof}
Recall that we set $(1,0,0)$ to be the multiplicative identity. We need not check the Leibniz rule for elements of the form $(q,0,0)$ for $q \in Q$, since the differential $\partial^{\text{Cone}(\Psi)}$ is $Q$-linear.

For the ease of the reader, we recall that for homogeneous $(f_1,g_1,g_2) \in \text{Cone}(\Psi)$ of total degree $i$ we have
\[
\partial^{\text{Cone}(\Phi)}(f_1,g_1,g_2) =
\begin{cases}
    (\partial^{\FF}(f_1)+\partial^{\GG}(g_1), 0, 0), & i=1,\\
    (\partial^{\FF}(f_1)+\Phi(g_2), \partial^{\GG}(g_1)-zg_2, 0), & i=2,\\
    (\partial^{\FF}(f_1)+\Phi(g_2), \partial^{\GG}(g_1)-zg_2, -\partial^{\GG}(g_2)), &i \geq 3. 
\end{cases}
\]

The Leibniz rule holds for product \eqref{prod:ff}, $(f_1,0,0)(f_2,0,0)$, 
and follows from the dg structure of $\FF$:
    \begin{align*}
        \partial((f_1,0,0)(f_2,0,0)) &= \partial((f_1f_2,0,0)) \\
            &= (\partial^{\FF}(f_1f_2),0,0) \\
            &= (\partial^{\FF}(f_1)f_2,0,0)+ ((-1)^{|f_1|}f_1\partial^{\FF}(f_2),0,0) \\
            &= (\partial^{\FF}(f_1),0,0)(f_2,0,0) + (-1)^{|f_1|}(f_1,0,0)(\partial^{\FF}(f_2),0,0) \\
            &= \partial((f_1,0,0))(f_2,0,0)+ (-1)^{|f_1|}(f_1,0,0)\partial((f_2,0,0))
    \end{align*}

For product \eqref{prod:fg1}, $(f_1,0,0)(0,g_2,0)$, we consider four cases:

\textbf{Case 1:} Let $\partial(f_1)=x_iz$ and $\partial(g_2)=x_jy_{j,l}$.
        \begin{align*}
            \partial((f_1,0,0)(0,g_2,0)) &= \partial\left(\frac{1}{z}f_1\Phi(g_2),0,-x_ig_2\right) \\
                &= \left(\frac{1}{z}\partial^{\FF}(f_1 \Phi(g_2))+\Phi(-x_ig_2),-z(-x_ig_2),0\right) \\
                &= \left(\frac{x_i z}{z} \Phi(g_2) - \frac{1}{z}f_1\partial^{\FF}(\Phi(g_2)) - x_i \Phi(g_2), zx_i g_2\right) \\
                &= (-x_jy_{j,l}f_1,x_izg_{2},0) \\
                &= (x_iz,0,0)(0,g_2,0)-(x_jy_{j,l}f_1,0,0) \\
                &= (\partial^{\FF}(f_1),0,0)(0,g_2,0) + (-1)^{|f_1|}(f_1,0,0)(x_j y_{j,l},0,0) \\
                &= \partial((f_1,0,0))(0,g_2,0)+(-1)^{|f_1|}(f_1,0,0)\partial((0,g_2,0)).
        \end{align*}

\textbf{Case 2:} Let $\partial(f_1)=x_iz$ and $|g_2|>1$.
    \begin{align*}
        \partial((f_1,0,0)(0,g_2,0)) &= \partial\left(\frac{1}{z}f_1\Phi(g_2),0,-x_ig_2\right) \\
            &= \left(-\frac{1}{z} \partial^{\FF}(f_1 \Phi(g_2)) + \Phi(-x_i g_2), x_izg_2, x_i\partial^{\GG}(g_2)\right) \\
            &= \left(-\frac{1}{z} f_1 \partial^{\FF}(\Phi(g_2)), x_i z g_2, +x_i \partial^{\GG}(g_2)\right) \\
            &= (x_iz,0,0)(0,g_2,0) - \left(\frac{1}{z} f_1 \Phi(\partial^{\GG}(g_2)), 0, -x_i \partial^{\GG}(g_2)\right) \\
            &= (\partial^{\FF}(f_1),0,0)(0,g_2,0)-(f_1,0,0)(0,\partial^{\GG}(g_2),0) \\
            &= \partial((f_1,0,0))(0,g_2,0)+(-1)^{|f_1|}(f_1,0,0)\partial((0,g_2,0)).
    \end{align*}

\textbf{Case 3:} Let $|f_1|>1$, and $\partial(g_2)=x_jy_{j,l}$.
    \begin{align*}
        \partial((f_1,0,0)(0,g_2,0)) &= \partial\left(\frac{1}{z}f_1\Phi(g_2),0,0\right) \\
            &= \left(\frac{1}{z}\partial^{\FF}(f_1 \Phi(g_2)), 0, 0\right) \\
            &= \left(\frac{1}{z}\partial^{\FF}(f_1) \Phi(g_2), 0, 0\right) + (-1)^{|f_1|}\left(\frac{1}{z}f_1\partial^{\FF}(\Phi(g_2)),0,0\right) \\
            &= \left(\frac{1}{z}\partial^{\FF}(f_1) \Phi(g_2), 0, 0\right) + (-1)^{|f_1|}(x_jy_{j,l}f_1,0,0) \\
            &= (\partial^{\FF}(f_1),0,0)(0,g_2,0) + (-1)^{|f_1|}(f_1,0,0)(x_j y_{j,l},0,0) \\
            &= \partial((f_1,0,0))(0,g_2,0)+(-1)^{|f_1|}(f_1,0,0)\partial((0,g_2,0)).
    \end{align*}

\textbf{Case 4:} Let $|f_1|,|g_2|>1$.
    \begin{align*}
        \partial((f_1,0,0)(0,g_2,0)) &= \partial\left(\frac{1}{z}f_1\Phi(g_2),0,0\right) \\
            &= \left(\frac{1}{z} \partial^{\FF}(f_1\Phi(g_2)),0,0\right) \\
            &= \left(\frac{1}{z} \partial^{\FF}(f_1)\Phi(g_2),0,0\right) + (-1)^{|f_1|}\left(\frac{1}{z} f_1 \partial^{\FF}(\Phi(g_2)),0,0\right) \\
            &= \left(\frac{1}{z} \partial^{\FF}(f_1)\Phi(g_2),0,0\right) + (-1)^{|f_1|}\left(\frac{1}{z} f_1 \Phi(\partial^{\GG}(g_2)),0,0\right) \\
            &= (\partial^{\FF}(f_1),0,0)(0,g_2,0)+ (-1)^{|f_1|}(f_1,0,0)(0,\partial^{\GG}(g_2),0) \\
            &= \partial((f_1,0,0))(0,g_2,0)+(-1)^{|f_1|}(f_1,0,0)\partial((0,g_2,0)).
    \end{align*}

For product \eqref{prod:fg2}, $(f_1,0,0)(0,0,g_2)$, we check the same four cases. For these computations it is important to note that $|(0,0,g_2)|=|g_2|+1$.

\textbf{Case 1:} Let $\partial(f_1)=x_iz$ and $\partial(g_2)=x_jy_{j,l}$.
    \begin{align*}
        \partial((f_1,0,0)(0,0,g_2)) &= (0,0,0) \\
            &= (0,0,x_izg_2) - \left(f_1\Phi(g_2) - \frac{z}{z} f_1 \Phi(g_2),0, x_izg_2\right) \\
            &= (x_iz,0,0)(0,0,g_2)-(f_1,0,0)(\Phi(g_2),-zg_2,0)) \\
            &= \partial((f_1,0,0))(0,0,g_2)+(-1)^{|f_1|}(f_1,0,0)\partial((0,0,g_2)).
    \end{align*}

\textbf{Case 2:} Let $\partial(f_1)=x_iz$ and $|g_2|>1$.
    \begin{align*}
        \partial((f_1,0,0)(0,0,g_2)) &= (0,0,0) \\
            &= (0,0,x_izg_2) - \left(f_1 \Phi(g_2) - \frac{z}{z} f_1 \Phi(g_2), 0, x_izg_2\right) \\
            &= (x_iz,0,0)(0,0,g_2)-(f_1,0,0)(\Phi(g_2),-zg_2,-\partial^{\GG}(g_2))) \\
            &= \partial((f_1,0,0))(0,0,g_2)+(-1)^{|f_1|}(f_1,0,0)\partial((0,0,g_2)).
    \end{align*}

\textbf{Case 3:} Let $|f_1|>1$ and $\partial(g_2)=x_jy_{j,l}$.
    \begin{align*}
        \partial((f_1,0,0)(0,0,g_2)) &=(0,0,0) \\
            &= (0,0,0) + (-1)^{|f_1|}\left(f_1 \Phi(g_2) - \frac{z}{z} f_1 \Phi(g_2),0,0\right) \\
            &= (\partial^{\FF}(f_1),0,0)(0,0, g_2)+(-1)^{|f_1|}(f_1,0,0)(\Phi(g_2),-z g_2,0) \\
            &= \partial((f_1,0,0))(0,0,g_2)+(-1)^{|f_1|}(f_1,0,0)\partial((0,0,g_2)).
    \end{align*}

\textbf{Case 4:} Let $|f_1|,|g_2|>1$.
    \begin{align*}
        \partial((f_1,0,0)(0,0,g_2)) &= (0,0,0) \\
            &= (0,0,0) + (-1)^{|f_1|} \left(f_1 \Phi(g_2) - \frac{z}{z} f_1 \Phi(g_2),0,0\right) \\
            &= (\partial^{\FF}(f_1),0,0)(0,0,g_2)+(-1)^{|f_1|}(f_1,0,0)(\Phi(g_2),-zg_2,-\partial^{\GG}(g_2)) \\
            &= \partial((f_1,0,0))(0,0,g_2)+(-1)^{|f_1|}(f_1,0,0)\partial((0,0,g_2)).
    \end{align*}

Similar to the first product, the Leibniz rule holds for product \eqref{prod:gg}, $(0,g_1,0)(0,g_2,0)$, follows from the dg algebra structure of $\GG$ with a small caveat. We observe that
\begin{align*}
    \partial((0,g_1,0))(0,g_2,0) &= \begin{cases} (x_i y_{i,l},0,0)(0,g_2,0) & \partial(g_1) = x_i y_{i,l} \\ (0,\partial^{\GG}(g_1), 0) (0,g_2,0) & |g_1| > 1 \end{cases} \\
     &= \begin{cases} (0,x_i y_{i,l}g_2,0) & \partial(g_1) = \{x_i y_{i,l}\} \\ (0,\partial^{\GG}(g_1)g_2,0) & |g_1| > 1 \end{cases} \\
     &= (0,\partial^{\GG}(g_1)g_2,0)
\end{align*}
and thus we have the following:
\begin{align*}
    \partial((0,g_1,0)(0,g_2,0)) &= \partial((0,g_1g_2,0)) \\
        &= (0,\partial^{\GG}(g_1g_2),0) \\
        &= (0,\partial^{\GG}(g_1)g_2,0)+ (0,(-1)^{|g_1|}g_1\partial^{\GG}(g_2),0) \\
        &= \partial((0,g_1,0))(0,g_2,0)+ (-1)^{|g_1|}(0,g_1,0)\partial((0,g_2,0)).
\end{align*}

For product \eqref{prod:g1g2}, $(0,g_1,0)(0,0,g_2)$, recall $|(0,0,g_2)|=|g_2|+1$. We consider four cases. If $\partial(g_2) = x_jy_{j,k}$, then $\Phi(g_2)=y_{j,k}f_{\{x_jz\}}$ and $\partial^{\FF}(f_{\{x_jz\}})=x_jz$, hence
\[\omega_{\Phi(g_2)}(g_1) = y_{j,k}\omega_{f_{\{x_jz\}}}(g_1) = - y_{j,k} x_j g_1.\] We use this in cases 1 and 3 below.

\textbf{Case 1:} Let $\partial(g_1)=x_iy_{i,l}$ and $\partial(g_2)=x_jy_{j,k}$.
    \begin{align*}
        \partial((0,g_1,0)(0,0,g_2)) &= \partial(-(0, 0, g_1g_2)) \\
            &= (-\Phi(g_1g_2), zg_1g_2, \partial^{\GG}(g_1g_2)) \\
            &= \left(-\frac{1}{z} \Phi(g_1) \Phi(g_2), z g_1 g_2, \partial^{\GG}(g_1) g_2 - g_1 \partial^{\GG}(g_2)\right) \\
            &= (0,0,x_iy_{i,l}g_2) - \left(\frac{1}{z} \Phi(g_1) \Phi(g_2), -z g_1 g_2, -y_{j,k} x_j g_1\right) \\
            &= \partial_1((0,g_1,0))(0,0,g_2) + (-1)^{|g_1|}(0,g_1,0)\partial((0,0,g_2)).
    \end{align*}

\textbf{Case 2:} Let $\partial(g_1)=x_iy_{i,l}$ and $|g_2|>1$.
    \begin{align*}
        \partial((0,g_1,0)(0,0,g_2)) &= -\partial((0,0,g_1g_2)) \\
            &= - \left(\Phi(g_1g_2), -zg_1g_2, - \partial^{\GG}(g_1g_2)\right) \\
            &= \left(\frac{-1}{z} \Phi(g_1) \Phi(g_2), z g_1 g_2, \partial^{\GG}(g_1)g_2 - g_1 \partial^{\GG}(g_2)\right) \\
            &= (0,0,x_i y_{i,\ell}g_2) - \left(\frac{1}{z} \Phi(g_1) \Phi(g_2), - z g_1 g_2, g_1 \partial^{\GG}(g_2) \right) \\
            &= (x_iy_{i,l},0,0)(0,0,g_2)-(0,g_1,0)(\Phi(g_2),-zg_2,-\partial^{\GG}(g_2)) \\
            &= \partial_1((0,g_1,0))(0,0,g_2) + (-1)^{|g_1|} (0,g_1,0)\partial((0,0,g_2)).
    \end{align*}

\textbf{Case 3:} Let $|g_1|>1$ and $\partial(g_2)=x_jy_{j,k}$.
    \begin{align*}
        \partial((0,g_1,0)(0,0,g_2)) &= (-1)^{|g_1|}\partial(0,0,g_1g_2) \\
            &= (-1)^{|g_1|}(\Phi(g_1g_2),-zg_1g_2,-\partial^{\GG}(g_1g_2)) \\
            &= (-1)^{|g_1|} \left(\frac{1}{z} \Phi(g_1) \Phi(g_2), - z g_1 g_2, -\partial^{\GG}(g_1)g_2 - (-1)^{|g_1|} g_1 \partial^{\GG}(g_2)\right) \\
            &= (0,\partial^{\GG}(g_1),0)(0,0,g_2) - (-1)^{|g_1|}\left(\frac{-1}{z} \Phi(g_1) \Phi(g_2), z g_1 g_2, (-1)^{|g_1|} y_{j,k}x_j g_1 \right) \\
            &= (0,\partial^{\GG}(g_1),0)(0,0,g_2)+(-1)^{|g_1|}(0,g_1,0)(\Phi(g_2),-zg_2,0) \\
            &= \partial((0,g_1,0))(0,0g_2)+ (-1)^{|g_1|}(0,g_1,0)\partial((0,0,g_2)).
    \end{align*}

\textbf{Case 4:} Let $|g_1|,|g_2|>1$.
    \begin{align*}
        \partial((0,g_1,0)(0,0,g_2)) &= (-1)^{|g_1|}\partial((0,0,g_1g_2)) \\
            &= (-1)^{|g_1|}(\Phi(g_1g_2),-zg_1g_2,-\partial^{\GG}(g_1g_2)) \\
            &= (-1)^{|g_1|} \left(\Phi(g_1 g_2), - z g_1 g_2, - \partial^{\GG}(g_1)g_2 -(-1)^{|g_1|} g_1 \partial^{\GG}(g_2)\right) \\
            &= (0,\partial^{\GG}(g_1),0)(0,0,g_2) - (-1)^{|g_1|}\!\left(\frac{-1}{z} \Phi(g_1)\Phi(g_2), z g_1 g_2, (-1)^{|g_1|} g_1 \partial^{\GG}(g_2)\right) \\
            &= (0,\partial^{\GG}(g_1),0)(0,0,g_2)+(-1)^{|g_1|}(0,g_1,0)(\Phi(g_2),-zg_2,-\partial^{\GG}(g_2)) \\
            &= \partial((0,g_1,0))(0,0g_2)+ (-1)^{|g_1|}(0,g_1,0)\partial((0,0,g_2)).
    \end{align*}

For product \eqref{prod:shiftg}, $(0,0,g_1)(0,0,g_2)$, recall $|(0,0,g_1)| = |g_1| + 1$ and $|(0,0,g_2)|=|g_2|+1$. We consider our four cases.

\textbf{Case 1:} Let $\partial(g_1)=x_iy_{i,l}$ and $\partial(g_2)=x_jy_{j,k}$.
    \begin{align*}
        \partial((0,0,g_1)(0,0,g_2)) &= (0,0,0) \\
            &= -(0,0,-zg_1g_2)+(0,0,-zg_1g_2) \\
            &= (\Phi(g_1),-zg_1,0)(0,0,g_2)+(0,0,g_1)(\Phi(g_2),-zg_2,0) \\
            &= \partial((0,0,g_1))(0,0,g_2)+(-1)^{|g_1|+1}(0,0,g_1)\partial((0,0,g_2)).
    \end{align*}
        
\textbf{Case 2:} Let $\partial(g_1)=x_iy_{i,l}$ and $|g_2|>1$.
    \begin{align*}
        \partial((0,0,g_1)(0,0,g_2)) &= (0,0,0) \\
            &= -(0,0,-zg_1g_2) + (0,0,-zg_1g_2) \\
            &= (\Phi(g_1),-zg_1,0)(0,0,g_2)+(0,0,g_1)(\Phi(g_2),-zg_2,-\partial^{\GG}(g_2)) \\
            &= \partial((0,0,g_1))(0,0,g_2)+ (-1)^{|g_1|+1}(0,0,g_1)\partial((0,0,g_2)).
    \end{align*}

\textbf{Case 3:} The case $|g_1|>1$ and $\partial(g_2)=x_jy_{j,k}$ follows from the previous using graded commutativity.

\textbf{Case 4:} Let $|g_1|, |g_2| >1$.
    \begin{align*}
        \partial((0,0,g_1)(0,0,g_2)) &= (0,0,0) \\
            &= (-1)^{|g_1|}(0,0,-zg_1g_2) - (-1)^{|g_1|}(0,0,-zg_1g_2) \\
            &= (\Phi(g_1),-zg_1,-\partial^{\GG}(g_1))(0,0,g_2)-(-1)^{|g_1|}(0,0,g_1)( \Phi(g_2),-zg_2,-\partial^{\GG}(g_2)) \\
            &= \partial((0,0,g_1))(0,0,g_2)+ (-1)^{|g_1|+1}(0,0,g_1)\partial((0,0,g_2)).
    \end{align*}

Product \eqref{prod:g1f} follows from product \eqref{prod:fg1} using graded commutativity. Similarly, products \eqref{prod:g2f} and \eqref{prod:g2g1} follow from \eqref{prod:fg2} and \eqref{prod:g1g2}, respectively.

\end{proof}

\section*{Declarations on Funding, Competing Interests, and Data Availability}

No funds, grants, or other support was received for conducting this research, and the authors have no relevant financial or non-financial interests to disclose.  No data was used in the process.  This work represents the views of the authors and is not to be regarded as representing the opinions of the Center for Naval Analyses or any of its sponsors.  

\printbibliography

\end{document}